\documentclass[oneside, a4paper,11pt,reqno]{amsart}

\usepackage[T1]{fontenc}

\usepackage{verbatim}
\usepackage{amsmath}
\usepackage{amsfonts}
\newtheorem{theo}{Theorem}
\newtheorem{prop}[theo]{Proposition}

\newtheorem{lemma}[theo]{Lemma}
\newcommand {\pare}[1] {\left( {#1} \right)}
\newcommand {\cro}[1] {\left[ {#1} \right]}
\newcommand {\acc}[1] {\left\{ {#1} \right\}}
\newcommand {\nor}[1] { \left\| {#1} \right\|}
\newcommand {\bra}[1] { \langle {#1} \rangle}
\newcommand {\ent}[1] { \lfloor {#1} \rfloor}

\def \E {\mathbb{E}}
\def \P {\mathbb{P}}
 
\def \R  {\mathbb{R}} 
\def \T  {\mathbb{T}} 
\def \Z  {\mathbb{Z}}
\def \C  {\mathbb{C}}

\def \FF {{\mathcal F}}
\def \EE {{\mathcal E}}
\def \II {{\mathcal I}}
\def \MM {{\mathcal M}}
\def \ind {\hbox{ 1\hskip -3pt I}}
\newcommand {\va}[1] {\left| {#1} \right|}
\newcommand {\floor}[1] {\lfloor {#1}\rfloor}

\newtheorem{remarque}[theo]{Remark}
\def\F {\mathcal{F}}
\newcommand {\abs}[1] {\left\lvert {#1} \right\rvert}
\newcommand {\refeq}[1] {(\ref{#1})}

\begin{document}

\title[Large Deviations for stable random walks self intersection local times.]
{Exponential moments of self-intersection local times of stable random walks in subcritical dimensions}
\author{Fabienne Castell} 
\address{LATP, UMR CNRS 7353. Centre de Math\'ematiques et Informatique.
 Aix-Marseille Universit\'e. 39, rue Joliot Curie. 13 453 Marseille Cedex
13. France.}
\email{Fabienne.Castell@cmi.univ-mrs.fr}

\author{Cl\'ement Laurent} 
\address{LATP, UMR CNRS 7353. Centre de Math\'ematiques et Informatique.
 Aix-Marseille Universit\'e. 39, rue Joliot Curie. 13 453 Marseille Cedex
13. France.}
\email{Clement.Laurent@cmi.univ-mrs.fr}

\author{Clothilde M\'elot} 
\address{LATP, UMR CNRS 7353. Centre de Math\'ematiques et Informatique.
 Aix-Marseille Universit\'e. 39, rue Joliot Curie. 13 453 Marseille Cedex
13. France.}
\email{Clothilde.Melot@cmi.univ-mrs.fr}

\keywords{Self-intersection local times; Random walk in random scenery; large deviations; stable random walk\\
This research was supported by the french ANR project MEMEMO2} 

\begin{abstract}
Let $(X_t, t \ge 0)$ be an $\alpha$-stable random walk with values in $\Z^d$. Let 
$l_t(x) = \int_0^t \delta_x(X_s) ds$ be its local time.  For $p>1$, not necessarily integer, 
$I_t = \sum_x l_t^p(x)$ is the so-called $p$-fold self- intersection local time of the random walk.  
When  $p(d -\alpha) < d$,  we derive precise logarithmic asymptotics of the probability  $\P\cro{I_t \ge r_t}$ for 
all scales $r_t  \gg \E\cro{I_t}$. Our result extends previous works by Chen, Li \& Rosen 2005, Becker \& K\"onig 2010, and Laurent 2012.
\end{abstract}

\maketitle

\section{Introduction}
 Intersections of random paths have been an extensively studied topic,  not only due to its fundamental importance
  in the theory of stochastic processes, but also because they play a central role in various physical models  as the "polaron problem" in quantum field theory \cite{DV}, the parabolic Anderson model describing diffusion in random potential (\cite{KMSS,ZMRS,CM}), or 
  random polymer models (\cite{DJ,Ed,PdG}). Therefore, various mathematical objects have been introduced to quantify 
  these intersections: the range of a random walk, the volume of the Wiener sausage, or the self-intersection local times, which is the 
  main object of the present paper. 
 
 \subsection{Self-intersection local times.}
 Let $(X_t, t \ge 0)$ be a continous time random walk on $\Z^d$, with jump rate 1, and generator
 $A$: 
 \[  Af(x) = \sum_{y \in \Z^d} \mu(y-x)(f(y)- f(x)), 
 \] 
where $\mu$ denotes the  law of the increments of the random walk. Throughout the paper, we assume 
 that $\mu$  is in the domain of attraction of a stable law of index $\alpha \in (0;2]$ and that $\mu$  is symmetric. 
 More precise assumptions on $\mu$ will be given later. 
For any $x \in \Z^d$, let $l_t(x) = \int_0^t \delta_{x}(X_s) \, ds$ be the time spent by the random walk on site
$x$ up to time $t$. For a positive real number $p \ge 1$, consider the $\ell_p$ norm of $l_t$ 
\[ N_p(l_t) = \pare{\sum_{x \in \Z^d} l_t(x)^p}^{1/p} \, 
\]
Note that for $p=1$, $N_1(l_t)=t$, while for $p=0$, it is equal to the number of distinct sites visited by the random walk up to time 
$t$. Moreover, since $l_t(x) \le t$ for all $x \in \Z^d$, $N_p(l_t) \le t$. If 
$p$ is an integer,
\[  I_t:=N_p^p(l_t)= \int_0^t dt_1 \cdots \int_0^t dt_p \ind_{X_{t_1}=\cdots=X_{t_p}} \, ,
\]
is the so called $p$-fold intersection local time of the random walk, which measures the time spent 
by the random walk on sites visited at least $p$ times.  
  To begin our study of the large time behavior of $N_p(l_t)$, let us have a look at its typical behavior, which depends on the
  recurrence/transience property of the random walk. We state the results in terms of $I_t$: 
 \begin{itemize}
 \item {\bf Transient case ($d > \alpha$)}: in this case, the range of the random walk is of order $t$ and the random walk
 spends a time of order 1 on each visited site, so that $I_t$ is of order $t$.  In the introduction of \cite{KS}, Kesten \& Spitzer 
 proved that $I_t/t$ converges a.s. to some  deterministic constant. 
 \item {\bf Recurrent case ($d <\alpha$, i.e $d=1, \alpha >1$)}:   
  the range of the random walk is now of order $t^{1/\alpha}$; the random walk
 spends a time of order $t^{1-1/\alpha}$ on each visited site, so that $I_t$ is of order $t^{p-(p-1)/\alpha}$. More precisely, 
  $I_t/t^{p-(p-1)/\alpha}$ converges in distribution to the $L_p$-norm of the stable limiting process local time (see lemme 6 in \cite{KS}, 
  or lemma 14 in \cite{CLR}). 
  \item {\bf Critical case ($d=\alpha$; i.e $d=\alpha=1$, $d=\alpha=2$)}: the range of the random walk is now of order $t/\log(t)$; 
  the random walk spends a time of order $\log(t)$ on each visited site, so that $I_t$ is of order $t \log(t)^{p-1}$. More precisely, 
  $I_t/(t \log(t)^{p-1})$ converges a.s. to some explicit deterministic constant  (see \cite{cerny} or lemme 4 in \cite{CGP}). 
 \end{itemize}
 To summarize, we get 
 \begin{equation}
 \label{typical}
 \E (I_t) \asymp \begin{cases} t^{p-(p-1)/\alpha} & \mbox{ if } d < \alpha \, , \\
 							t \log(t)^{p-1} & \mbox{ if } d=\alpha \, , \\
							t & \mbox{ if } d > \alpha \, .
							\end{cases}
\end{equation}

\subsection{Main result.}
In this paper we are interested in the large deviations asymptotics of $N_p(l_t)$ when the walk  produces an excess of self-intersections.
More precisely, we give exact logarithmic asymptotics of $\P\cro{I_t \ge r_t}$ for scales $t^p \gg r_t \gg \E(I_t)$. 
 This problem has received a lot of attention during the last decade, and we refer the reader to the recent monograph \cite{chen},  or to
 the survey paper \cite{konig} for an up-to-date picture of known results and of the various technics used  in order to prove them. Let us
 stress that on a heuristic level, there is a phase transition in the optimal behavior of the random walk to produce many 
 self-intersections:
 \begin{itemize}
 \item In the supercritical case $d > \alpha q$ (where $q$ is the conjugate exponent of $p$), 
  this optimal behavior is  to stay confined in a ball of radius of order 1 during a time 
 period of order  $r_t^{1/p} \ll t$. This confinement happens with probability of order $\exp(-r_t^{1/p})$. 
 Rough logarithmic asymptotics in the supercritical case were first proved in \cite{AC} (for $p=\alpha=2$) and later
 refined in precise logarithmic asymptotics in \cite{A,ClemSPA}.
 \item   In the subcritical case $d < \alpha q$, the optimal behavior for the random walk is to remain up to time $t$ in
 a ball whose volume is of order $t^q r_t^{-q/p} \gg 1$.  This confinement happens with probability of order $\exp(- t^{1-(\alpha q)/d}
 r_t^{(\alpha q)/(pd)})$. Precise logarithmic asymptotics were first proved in $d=1$ in \cite{CL} for $\alpha =2$, and later
 extended in \cite{CLR} for $\alpha > 1$. Very recently, the case $d \ge 2, \alpha =2$ was treated in \cite{BK} (with the
 restriction $d < 2/(p-1) < 2q$), and \cite{ClemEJP}.
 \item In the critical case $d=\alpha q$, any confinement strategy in a ball of volume $R^d \in [1; t^q r_t^{-q/p}]$
 during a time period of order  $r_t^{1/p} R^{d/q}$ has a probability of order $\exp(-r_t^{1/p})$, so that the intuitive picture  
 is not clear. Nevertheless, precise logarithmic asymptotics were established in \cite{C} for $\alpha=2$, and later
 extended in \cite{ClemSPA} to other values of $\alpha$.
 \end{itemize}
 The main result of this paper is to prove logarithmic asymptotics of $\P\cro{I_t \ge r_t}$ in the subcritical case, 
 for any value of $\alpha \in (0,2]$, thus extending \cite{CLR} to $d \ge 2$ and $\alpha<1$, and \cite{BK,ClemEJP}
 to $\alpha < 2$. To  state our result,  we have to introduce 
some notations and make precise assumptions on $\mu$.
 For $p>1$, $L^p (\R^d)$ is the usual Sobolev space of $p$-integrable functions w.r.t Lebesgue measure.
 $\nor{ \, }_p$ is the norm on  $L^p (\R^d)$.  For any function $g$ in $L^2(\R^d)$, we denote by $\F(g)$ its Fourier transform:
\begin{equation}
\label{Fourier}
 \F(g)(\omega)= \int_{\R^d} g(x) \exp(-2i\pi \bra{x,\omega}) \, dx \, , \omega \in \R^d .
\end{equation}
Throughout the paper, we assume that 
\begin{description}
\item[(H1)]  $\mu$ is symmetric.
\item[(H2)]  $\mu$ is in the normal domain of attraction of a strictly (symmetric) stable law $S_{\alpha}$
with index $\alpha \in (0,2]$ with characteristic function  given by $\exp(-\va{t}^{\alpha})$. This means that there
exists a constant $c$ such that $(\frac{c}{t^{1/\alpha}} X_{st}, s \in [0;1])$ converges in distribution when $t \rightarrow +
\infty$,
in the space $D([0,1])$ of c\`adl\`ag functions, endowed with $J_1$-topology. Without loss of 
generality, we will assume that $c=2\pi$. This choice is made in order to get nice statements
 using \refeq{Fourier} as a definition of Fourier transform.
\item[(H3)] $\mu$ is non-arithmetic, i.e $\acc{\omega \in \R^d, \text{ such that } 
\va{\sum_{x \in \Z^d} \exp(2i\pi \bra{\omega, x})
\mu(x)}=1} = \Z^d$
\item[(H4)]   Let $p_t(x,y)$ be the transition probabilities of $(X_t, t \ge 0)$. There exists a constant $\kappa$ such
that for any $t > 0$,  and any $x, y \in \Z^d$, 
\begin{equation}
\label{est-pt}
 p_t(x,y) \le \kappa \pare{ \frac{1}{t^{d/\alpha}} \wedge \frac{t}{\va{x-y}^{d+\alpha}} },
\end{equation}
 where $x\wedge y=\text{min}(x,y)$. 
When $\alpha = 2$, this is true by standard Gaussian estimates. When $\alpha < 2$, assumption {\bf (H1)} and 
the following condition on $\mu$ are shown
to imply \refeq{est-pt} in \cite{BL02}: there are constants $c_1$, $c_2$, such that for any $x\in \Z^d$, $x\ne 0$,
\[  \frac{c_1}{\va{x}^{d+\alpha}} \le \mu(x) \le \frac{c_2}{\va{x}^{d +\alpha}} \, .
\]
\end{description}

We next define 
\begin{equation}
\label{rho}
\rho_{\alpha,d,p} = \pare{\frac{\alpha q}{\alpha q-d}}\pare{\frac{\alpha q-d}{d}}^{d/\alpha q} \hspace{-.5cm}
  \inf_{g:\R^d\rightarrow \R}\acc{\nor{g}_2^{2\pare{1-d/\alpha q}}   \pare{ \int_{\R^d}\vert \omega\vert^\alpha \vert\F{(g)}(\omega)\vert^2d\omega}^{d/\alpha q}  \hspace{-.5cm},\nor{g}_{2p}=1}.
\end{equation}
\begin{equation}
\label{chi}
\chi_{\alpha,d,p}:=
\inf_{g:\R^d\rightarrow \R}\acc{\int\limits_{\R^d}\vert \omega\vert^\alpha \vert\F{(g)}(\omega)\vert^2d\omega,\nor{g}_{2p}=\nor{g}_2=1} .
\end{equation}
Note that the expression in the infimum appearing in \refeq{rho} is invariant under the transformation $g \mapsto g_{\lambda}=
\lambda^{d/(2p)} g(\lambda \, \cdot \,)$, so that one can reduce this infimum to functions $g$ satisfying $\nor{g}_{2p}=
\nor{g}_2=1$. Therefore, 
\[ \rho_{\alpha,d,p} = \frac{\alpha q}{\alpha q -d} \pare{\frac{\alpha q-d}{d} \chi_{\alpha,d,p}}^{d/(\alpha q)} \, .
\]
 It is proved in  \cite{CR05} that these constants are not degenerate when $d  < \alpha q$. We have then the following result.

\begin{theo}
\label{OKCAVA3}
Assume {\bf (H1--4)}, and that $p>1$,  $\alpha \in (0;2]$ and $d<\alpha q$, where $q$ is the conjugate exponent
of $p$ ($1/p + 1/q =1$) . Let $(\beta_t, t \ge 0)$ be such that for large
$t$,  \\
-  $1\ll\beta_t^\alpha\ll t$ when $d<\alpha$,\\
-  $1\ll\beta_t^d\ll \frac{t}{\log(t)}$ when $d=\alpha$,\\
-  $1\ll \beta_t^d\ll t$ when for $d>\alpha$. \\
Then for any $\theta \ge 0$, 
\begin{equation}
\label{loglapNp}
\lim_{t\rightarrow +\infty} \frac{\beta_t^\alpha}{t} \log \E\cro{\exp\pare{\theta\beta_t^{d/q-\alpha}N_p(l_t)}}= \pare{\frac{\theta}{\rho_{\alpha,d,p}}}^\frac{\alpha q}{\alpha q-d} 
\end{equation}
and
\begin{equation}
\label{PGDSILT}
\lim\limits_{t\rightarrow +\infty} \frac{\beta_t^{\alpha}}{t} \log \P\pare{I_t\geq \theta t^p\beta_t^{-d(p-1)}}=-\theta^{\alpha/(d(p-1))}\chi_{\alpha,d,p}.
\end{equation}
\end{theo}

\begin{remarque}
 Denoting $r_t = t^p\beta_t^{-d(p-1)}$, we stress the fact that the conditions on $\beta_t$ appearing in theorem \ref{OKCAVA3}
are equivalent to $t^p \gg r_t \gg \E(I_t)$. Moreover, using the fact that the function $\theta \mapsto \va{\theta}^{(\alpha q)/(\alpha
q-d)}$ is essentially smooth, \refeq{PGDSILT} is a straightforward consequence of \refeq{loglapNp} and 
the G\"artner-Ellis theorem for positive random
variables large deviations principle (see Theorem 1.2.3 in \cite{chen}), and is equivalent to a full large deviations
principle for $t^{-p} \beta_t^{d(p-1)} I_t$ (see Theorem 1.2.1 in \cite{chen}).
\end{remarque}

\subsection{Comments on the proof. Literature remarks.}
   To understand the main difficulty in obtaining Theorem \ref{OKCAVA3}, let us introduce for any $x \in \R^d$,
\[ L_t(x) = \frac{\beta_t^d}{t} l_t(\lfloor \beta_t x \rfloor) \, .
\]
Note that $L_t$ is an element of the  set $\MM_1(\R^d)$  of probability measures on $\R^d$. Following Donsker-Varadhan's 
results, it is easy to check that $L_t$
satifies a weak large deviations principles on $\MM_1(\R^d)$, endowed with the weak convergence defined by duality 
with compactly supported continuous functions (see section 5 Theorem 11). The speed of this large deviations principle is $t/\beta_t^{\alpha}$, while its rate 
function is defined by
\[ \II(\mu) = \begin{cases}
\int_{\R^d} \va{\omega}^{\alpha} \va{\FF(g)(\omega)}^2 \, d \omega 
& \mbox{ if } d \mu = g^2 dx ,
\\
+ \infty & \mbox{ otherwise. }
\end{cases}
\]
 Moreover $\nor{L_t}_p^p = \beta_t^{d(p-1)} t^{-p} I_t$. Applying the contraction principle to $\mu   \in \MM^1(\R^d) 
 \mapsto \nor{\frac{d \mu}{dx}}_{p}$ would lead to Theorem \ref{OKCAVA3}. The main problem here is that the latter function is lower semicontinuous but not continuous in weak topology. We can therefore deduce the lower bounds in Theorem \ref{OKCAVA3} by
 contraction, but this is no more the case for the upper bounds. To circumvent this problem, various strategies have been proposed. 
 One can first try to smoothen $L_t$. Given a regularizing kernel $\phi_{\epsilon}$, the map $\mu   \in \MM^1(\R^d) 
 \mapsto \nor{\phi_{\epsilon} \star \frac{d \mu}{dx}}_{p}$ is now continuous in weak topology, so that 
 we just have to prove that $\nor{\phi_{\epsilon} \star L_t}_p$ and   $\nor{L_t}_p$ are exponentially close.
  Typically, this can   be done in 
 "small" dimension using the regularity of local times. 
 Roughly speaking, this is the way followed by Bass, Chen \& Rosen in a series of works starting with
 the paper of Chen \& Li \cite{CL} about the large deviations of the self-intersections ($p=2$) of Brownian motion or random
 walk in $d=1$. Later, they studied the cases $d=1, \alpha > d$ in \cite{CLR}, $p=d=\alpha=2$ in \cite{BCR}, $p=2$ and
 $d \ge \alpha > 2d/3$ in \cite{BCR2}. In $d=1$, they worked directly on the self-intersections of the limiting stable
 process, and used the existence and the regularity of their local times, thus reducing their result to $\alpha > d$. 
 In $d \ge 2$, they first proved large deviations results for intersections of $p$-independent processes or random walks
 using a regularisation procedure in \cite{C04,CR05}, and then transfered these results to large deviations for self-intersections. Here the 
 assumption $p=2$ is crucial since this tranfer is done via the bisection method introduced by Varadhan in \cite{V}, which
 does not work for $p \ne 2$.  Becker and K\"onig \cite{BK} proved Theorem \ref{OKCAVA3} for  $\alpha=2$, $p< 2 q/p$ ($ < 2q$) and
 $\beta_t \ll (t/ \log(t))^{d/(d+2)}$, using an upper bound for the joint density of the local times of a Markov chain 
 obtained in \cite{BHK}. Unfortunately, this upper bound is not precise enough to obtain the result in its full extent. We propose
 here to use a method of proof  introduced in \cite{C} in the critical case, and successfully extended in \cite{ClemSPA} to 
 the supercritical case, and to the subcritical case for $\alpha=2$ in \cite{ClemEJP}.  One of the main tool in the proof is the Dynkin's isomorphism theorem, according to which the law of the local times of a symmetric recurrent Markov process stopped at an independent exponential time, is related to the law of the square of a Gaussian process whose covariance function is the Green kernel of the
stopped Markov process. This allows to control the exponential moments of $N_p(l_t)$ by the
exponential moments of  $N_{2p,R\beta_t}(Z)^2$  where 
 \begin{itemize}
 \item  $N_{2p,R\beta_t}(f)$ is the $\ell_{2p}$-norm  of a real valued function $f$ defined on $\T_{R \beta_t}$ the discrete torus
 of radius $R \beta_t$; 
 \item $(Z_x, x \in  \T_{R \beta_t})$ is a centered Gaussian process independent of $(X_t, t \ge 0)$,  whose covariance function is given by 
  the Green kernel $G_{R \beta_t, \lambda_t}(x,y)$ of  $(X_t, t \ge 0)$ projected on $\T_{R \beta_t}$,
  stopped at an independent exponential time with parameter $\lambda_t = a \beta_t^{- \alpha}$. 
  \end{itemize}
  Loosely speaking, this leads to a non asymptotic upper bound
  \[ \P (N_p(l_t) \ge \theta^{1/p} t \beta_t^{-d/q}) \le \exp(a t \beta_t^{-\alpha} (1+o(1)) 
   \P (N_{p,R\beta_t}(Z^2) \ge 2 \theta^{1/p} t \beta_t^{-d/q}) ,
   \]
   where the first term in the right hand side comes from the stopping at an exponential time. 
As soon as the median of $N_{p,R\beta_t}(Z^2)$ is negligible with respect to  $t \beta_t^{-d/q}$ (which is equivalent
to the conditions on $\beta_t$),  concentration inequalities for norms of Gaussian processes yields 
\[ \P (N_p(l_t) \ge \theta^{1/p} t \beta_t^{-d/q}) \le \exp(a t \beta_t^{-\alpha} (1+o(1))  
\exp(- \theta^{1/p} t \beta_t^{-d/q} \sigma^{-2}) ,
\]
where $\sigma^2$ is the maximal variance of the process $Z$ viewed as an element of $l_{2p}(\T_{R\beta_t})$:
\[ \sigma^2 = \sup_{f:\T_{R\beta_t}\rightarrow \R} \acc{\bra{f, G_{R\beta_t, \lambda_t}  f}_{R\beta_t}; N_{(2p)', R \beta_t}(f)=1}
\, ,
\]
where $\langle\cdot ,\cdot\rangle_{R\beta_t}$ is the scalar product on $l^2(\T_{R\beta_t}),$
and $ G_{R\beta_t, \lambda_t} f(x)=\sum_{y\in\T_{R\beta_t}} G_{R\beta_t, \lambda_t}(x,y) f(y)$ for any $f\in l^2(\T_{R\beta_t}) $ and any $x\in \T_{R\beta_t}$. 
Since $G_{R\beta_t, \lambda_t} =( \lambda_t \mbox{Id} - A_{R \beta_t})^{-1}$ (where $A_{R \beta_t}$ is the generator
of the random walk on the torus), we get  
 \begin{equation}
 \label{NAUB}
  \P (N_p(l_t) \ge \theta^{1/p} t \beta_t^{-d/q}) \le \exp(- t \beta_t^{-\alpha} (- a + \theta^{1/p} 
 \beta_t^{\alpha - d/q} \rho(a,R,t) + o(1)) )
 \, , 
 \end{equation}
 with

%
%
%
%
%
\begin{equation}
\label{rhoart}
\rho(a,R,t)= \inf_{h:\T_{R\beta_t}\rightarrow \R}\acc{a\beta_t^{-\alpha} N_{2,R\beta_t}^2(h)  - <h,A_{R\beta_t} h>_{R\beta_t}, N_{2p,R\beta_t}(h)=1}.
\end{equation}
Up to this point, the proof is the same in the critical case, the supercritical case, or the subcritical case. 
  We just have to use the correct scalings. It remains now
to study the limit of the upper bound \refeq{NAUB} when $t$ goes to infinity. This is where the proofs differ. In the 
critical case and the supercritical case, the limiting constant is still given by a variational formula involving functions
defined on the grid $\Z^d$ (see \cite{C, ClemSPA}), while on the subcritical case, the limiting
constant is given by a variational formula involving functions defined on $\R^d$. Therefore, we have to interpolate
the minimisers in $\rho(a,R,t)$. When $\alpha=2$, the operator $A$ is local, and this interpolation is done via
linear interpolation (see \cite{ClemEJP}). This proof does not work anymore when $\alpha < 2$, and we 
use here interpolation via Fourier transform. 

\vspace{.5cm}
The paper is organized as follows. Section \ref{prelim} is a reminder of classical results in Fourier analysis. 
Section \ref{ExpUB} is devoted to the upper bound in \refeq{loglapNp}. 
It relies on  Proposition \ref{OKCAVA4} giving the asymptotic behavior of $\rho(a,R,t)$, and Lemma \ref{G} about estimates 
of  $G_{R \beta_t, \lambda_t}(0,0)$ whose proofs are presented  in Section \ref{limrho}. 
Finally, we prove the lower bound in \refeq{loglapNp} in Section \ref{ExpLB}.

%

\section{Preliminaries results on Fourier transforms}
\label{prelim}
We gather in this section notations and well known results about Fourier transform used throughout the paper.  
\subsection{Fourier transform on the discrete torus.}
Let $R$ be an integer and let  $u: \T_R \mapsto \C$ be a function defined on the discrete 
$d$-dimensional torus $\T_R$ 
of radius $R$. We denote by 
$F_R\pare{u}:\T_R \mapsto \C$ its Fourier transform:  
\begin{equation*}
\forall n \in \T_R, F_{R}\pare{u}(n)=\sum_{k \in \T_R} u(k) \exp\pare{-2i\pi \frac{<k,n>}{R}}.
\end{equation*}
We have an inversion formula and a Parseval formula:
\begin{prop}
\label{inversion finie}
For $u: \T_R \mapsto \C$,
\begin{align*}
&\forall k \in \T_R, u(k)=\frac{1}{R^d}\sum_{n \in \T_R}F_R\pare{u}(n)\exp\pare{2i\pi\frac{<k,n>}{R}}  \\
&\text{and }\sum_{n \in \T_R}\va{F_R\pare{u}(n)}^2=R^d\sum_{k \in \T_R}\va{u(k)}^2.
\end{align*}
\end{prop}
\subsection{Fourier transform on $\Z^d$.}
For a sequence $u\in l_2(\Z^d)$, we denote by $F(u)$ its Fourier transform. $F\pare{u}$
is the 1-periodic function defined by
\begin{equation*}
\forall \omega\in \R^d, F\pare{u}(\omega)=\sum_{z\in\Z^d}u(z)\exp\pare{-2i\pi<z,\omega>}.
\end{equation*}
 There is an inversion and Parseval's formulas:
\begin{prop}
\label{généralisation parseval}
\begin{equation*}
u(z)=\int_{[-1/2,1/2]^d} F\pare{u}(\omega)\exp(2i\pi<z,\omega>)d\omega
\text{ and }
 \nor{F\pare{u}}_{2,[-1/2,1/2]^d}^2=N_{2}^2(u)
\end{equation*}
\end{prop}
These two types of Fourier transforms are linked through periodization: if  $\forall x\in\T_R$, 
$g_R(x)=\sum_{z\in\Z^d}g(x+Rz)$ is the periodised version of $g$, then 
\begin{equation}
\label{lien Fourier fini infini}
\forall x\in\T_R,  F_R\pare{g_R}(x)=F\pare{g}\pare{\frac{x}{R}}.
\end{equation}

\subsection{Fourier transform on the torus.}
For $g_R$ a $R$-periodic function on $\R^d$, we consider its Fourier coefficients:
\begin{equation*}
\forall n\in\Z^d, \F_R(g_R)(n)=\frac{1}{R^d}\int_{[0,R]^d} g_R(x)\exp\pare{-2i\pi\frac{<n,x>}{R}}dx.
\end{equation*}
Denoting by $\nor{ \cdot}_{2,R}$ the norm in $L^2([0,R]^d)$, we get the following inversion and Parseval's formulas:
\begin{prop}
\label{inversion périodique}
\begin{equation*}
\forall x\in\R^d,\ g_R(x)=\sum_{n\in\Z^d}\F_R(g_R)(n)\exp\pare{2i\pi \frac{<n,x>}{R}}
\text{ and } N_2^2\pare{\F_R(g_R)}=\frac{1}{R^d}\nor{g_R}_{2,R}^2.
\end{equation*}
\end{prop}

\subsection{Fourier transform on $\R^d$.}
When $g: \R^d \mapsto \R$ is an element of $L^2(\R^d)$, we define its Fourier transform $\F(g): \R^d \mapsto \C$
by equation \refeq{Fourier}. With this definition, Parseval's identity and inversion formula read 
\[
\nor{g}_2 = \nor{\F(g)}_2 \, , \,\,\, g(x) = \int_{\R^d} \F(g)(\omega) \exp(2i\pi \bra{x,\omega}) \, d\omega \, , x \in \R^d .
\]

To end  this section, we remind the reader the following Young inequalities: 
\begin{prop}
\label{inégalité Fourier}
let $p>2$ and $q$ be  its conjugate exponent.
For  any $u: \T_R \mapsto \C$,
\begin{equation*}
N_{p,R}(u)\leq R^{-d/q} N_{q,R}(F_R(u)). 
\end{equation*}
For any $R$-periodic function $g_R: \R^d \mapsto \C$,
\begin{equation*}
\nor{g_R}_{p,R}\leq R^{d/p} N_{q}(\F_R(g_R)).
\end{equation*}
\end{prop}
These inequalities are straightforward applications of the Riesz-Thorin interpolation theorem (see Theorem 1.3.4 in \cite{Gra} for example) and Parseval's formula.

\section{Exponential moments upper bound} 
\label{ExpUB}

This section is devoted to the proof of the upper bound in \refeq{loglapNp}.

Let $a > 0$ and
 $\tau$ be an exponential time of parameter $\lambda_t=a\beta_t^{-\alpha}$ independent of $(X_t, t \ge 0)$.
Let $(X_s^{R\beta_t}, s \ge 0)$ be the projection of $(X_t, t \ge 0)$ on the discrete torus $\T_{R \beta_t}$:
\[ X_s^{R\beta_t}= X_s \, \, \mbox{mod}(R \beta_t) \, .
\] 
Let $l_{R\beta_t,t}(x)=\int_0^{t} 
\delta_x (X_s^{R\beta_t}) \, ds$ be its local time at site $x$ up to time $t$.


\begin{align*}
N_p^p(l_t) &= \sum_{x \in \Z^d} l_t^p(x) 
	 = \sum_{x \in \T_{R\beta_t}} \sum_{k \in \Z^d} l^p_{t}(x+kR\beta_t)\\
    & \leq  \sum_{x \in \T_{R\beta_t}} \pare{\sum_{k \in \Z^d} l_t(x+kR\beta_t)}^p
      =\sum_{x \in \T_{R\beta_t}} l^p_{R\beta_t,t}(x) = N_{p,R\beta_t}^p(l_{R\beta_t,t}).
\end{align*}

Then using the fact that $\tau\sim\mathcal{E} (\lambda_t)$ is independent of $(X_s,s\geq 0)$ we get for all 
$\theta > 0$:
\begin{align}
 \E\cro{\exp\pare{\theta\beta_t^{d/q-\alpha}N_p(l_t)}}  \exp\pare{-t\lambda_t}
		&\leq \mathbb{E}\cro{\exp\pare{\theta\beta_t^{d/q-\alpha}N_{p,R\beta_t}(l_{R\beta_t,t})}; \tau\geq t}\\
\label{UB1}
		& \leq\mathbb{E}\cro{\exp\pare{\theta\beta_t^{d/q-\alpha}N_{p,R\beta_t}(l_{R\beta_t,\tau})}}
\end{align}

We are now going to use the following version of Dynkyn's isomorphism theorem:
\begin{theo} (Eisenbaum, see for instance corollary 8.1.2 page 364 in \cite{MR06}). \\
\label{Eisenbaum}
Let $\tau \sim \EE(\lambda_t)$ independent of $(X_s, s \geq 0)$, and
let $(Z_x, x \in \T_{R\beta_t})$ be a centered Gaussian process 
with covariance matrix $G_{R\beta_t,\lambda_t}(x,y)=E_x\cro{\int_0^\tau \delta_y(X_s^{R\beta_t})ds}$ 
independent of $\tau$  and of the  random walk $(X_s, s \geq 0)$. 
For $s\neq 0$, consider the process 
$S_x := l_{R\beta_t,\tau}(x) + \frac{1}{2} (Z_x+s)^2$. Then for all measurable
and bounded function $F : \R^{\T_{R\beta_t}} \mapsto \R$: 
\begin{equation*}
\E\cro{F((S_x; x\in \T_{R\beta_t}))}
= \E\cro{F\pare{(\frac{1}{2}(Z_x +s)^2;x \in \T_{R\beta_t})} \, \pare{1 + \frac{Z_0}{s}}}
\, . 
\end{equation*} 
\end{theo} 

This theorem allows to compare the tail behavior of $N_{p,R\beta_t}(l_{R\beta_t,\tau})$ with 
the tail behavior of  $N_{2p,R\beta_t}(Z)$.  Indeed, using that 
\[   S_x^p\geq l^p_{R\beta_t,\tau}(x) + \pare{\frac{1}{2} (Z_x+s)^2}^p \, ,
\]
and the independence of $(Z_x,x\in\T_{R\beta_t})$ with the random walk $(X_s,s\geq 0)$ and the exponential time $\tau$, we get  for all $\epsilon >0$,  and all $y >0$, 

\begin{align}
 \nonumber & \P\pare{N_{p,R\beta_t}^p(l_{R\beta_t,\tau})\geq y^pt^p \beta_t^{d(1-p)}} 
 \P\pare{\frac{1}{2^p}N_{2p,R\beta_t}^{2p}(Z+s)\geq t^p \beta_t^{d(1-p)}\epsilon^p}\\
\leq & \nonumber \P\pare{N_{p,R\beta_t}^p(l_{R\beta_t,\tau})+\frac{1}{2^p}N_{2p,R\beta_t}^{2p}(Z+s)\geq t^p\beta_t^{d(1-p)}(y^p+\epsilon^p)}\\
\leq &\nonumber \P\pare{N_{p,R\beta_t}^p(S)\geq t^p\beta_t^{d(1-p)}(y^p+\epsilon^p)}\\
= &\label{PPE} \E\cro{\pare{1+\frac{Z_0}{s}}; \frac{1}{2^p}N_{2p,R\beta_t}^{2p}(Z+s)\geq t^p\beta_t^{d(1-p)}(y^p+\epsilon^p)}
\end{align}
where the last equality comes from Theorem \ref{Eisenbaum}.
Moreover by H\"older's inequality, for all $\epsilon >0$,
\begin{align}
\nonumber
&\E\cro{\pare{1+\frac{Z_0}{s}} ; \frac{1}{2^p}N_{2p,R\beta_t}^{2p}(Z+s)\geq t^p\beta_t^{d(1-p)}(y^p+\epsilon^p)}\\
\label{EexpE}\leq &\E\cro{\va{1+\frac{Z_0}{s}}^{1+1/\epsilon}}^{\epsilon/(1+\epsilon)}
\P\pare{N_{2p,R\beta_t}^{2p}(Z+s)\geq 2^pt^p\beta_t^{d(1-p)}(y^p+\epsilon^p)}^{1/(1+\epsilon)}.
\end{align}
Combining  (\ref{PPE}) and (\ref{EexpE}), we obtain that for all $a,\epsilon >0$,
\begin{align}
\nonumber
& \P\pare{N_{p,R\beta_t}(l_{R\beta_t,\tau})\geq yt\beta_t^{-d/q}} \\
\label{PexpEP}
& \leq \E\cro{\va{1+\frac{Z_0}{s}}^{1+1/\epsilon}}^\frac{\epsilon}{1+\epsilon}
\frac{\P\pare{N_{2p,R\beta_t}(Z+s)\geq \sqrt{2t}\beta_t^{-d/(2q)}(y^p+\epsilon^p)^{1/(2p)}}^\frac{1}{1+\epsilon}}
{\P(N_{2p,R\beta_t}(Z+s)\geq \sqrt{2t\epsilon}\beta_t^{-d/(2q)})}  .
\end{align}

Then using the fact that $Var(Z_0)=G_{R\beta_t,\lambda_t}(0,0)$, and Markov's inequality,
we obtain that for all $\gamma > 0$, 
\begin{align}
\nonumber
&\P\pare{N_{p,R\beta_t}(l_{R\beta_t,\tau})\geq yt\beta_t^{-d/q}} \\
\label{EEexp}
\leq &C(\epsilon)\pare{1+\frac{\sqrt{G_{R\beta_t,\lambda_t}(0,0)}}{s}}
\frac{\P\pare{N_{2p,R\beta_t}(Z+s)\geq \sqrt{2t}\beta_t^{-d/(2q)}(y^p+\epsilon^p)^{1/(2p)}}^{1/(1+\epsilon)}}
{\P(N_{2p,R\beta_t}(Z+s)\geq \sqrt{2t\epsilon}\beta_t^{-d/(2q)})} \, ,
\\ 
\leq &C(\epsilon)\pare{1+\frac{\sqrt{G_{R\beta_t,\lambda_t}(0,0)}}{s}}
\frac{\E\cro{\exp\pare{\frac{\gamma}{2}\beta_t^{d/q-\alpha}N_{2p,R\beta_t}^2(Z+s)}}^{1/(1+\epsilon)}}
{\P(N_{2p,R\beta_t}(Z+s)\geq \sqrt{2t\epsilon}\beta_t^{-d/(2q)})} 
\exp\pare{-\gamma t\beta_t^{-\alpha}\frac{(y^p+\epsilon^p)^{1/p}}{1+\epsilon}} \, .
\end{align}
Therefore, 
\begin{align*}
&  \mathbb{E}\cro{\exp\pare{\theta\beta_t^{d/q-\alpha}N_{p,R\beta_t}(l_{R\beta_t,\tau})}} 
\\ 
& = 1+\int_0^{+\infty}\P\pare{N_{p,R\beta_t}(l_{R\beta_t,\tau})\geq yt\beta_t^{-d/q}}\theta t \beta_t^{-\alpha}\exp(\theta t y\beta_t^{-\alpha})dy\\
& \leq 1+C(\epsilon)\pare{1+\frac{\sqrt{G_{R\beta_t,\lambda_t}(0,0)}}{s}} \frac{ \E\cro{\exp\pare{\frac{\gamma}{2}\beta_t^{d/q-\alpha}N_{2p,R\beta_t}^2(Z+s)}}
^{1/(1+\epsilon)}}{\P(N_{2p,R\beta_t}(Z+s)\geq \sqrt{2t\epsilon}\beta_t^{-d/(2q)})} 
\\
& \hspace*{2cm} 
\int_0^{+\infty} \theta t\beta_t^{-\alpha} e^{-t\beta_t^{-\alpha} y (\frac{\gamma}{1+\epsilon}-\theta)} dy.
\end{align*}

Note that the integral is finite if and only if $\theta<\frac{\gamma}{1+\epsilon}$. In this case, 
\begin{align*}
& \mathbb{E}\cro{\exp\pare{\theta\beta_t^{d/q-\alpha}N_{p,R\beta_t}(l_{R\beta_t,\tau})}}
\\
&\leq 1+C(\epsilon,\theta,\gamma)\pare{1+\frac{\sqrt{G_{R\beta_t,\lambda_t}(0,0)}}{s}}
\frac{ \E\cro{\exp\pare{\frac{\gamma}{2}\beta_t^{d/q-\alpha}N_{2p,R\beta_t}^2(Z+s)}}
^{1/(1+\epsilon)}}{\P(N_{2p,R\beta_t}(Z+s)\geq \sqrt{2t\epsilon}\beta_t^{-d/(2q)})}  \, .
\end{align*}

Choosing $s=\frac{ \epsilon\sqrt{2t}}{R^\frac{d}{2p}\beta_t^{d/2}} $, using triangle inequality and the fact that $N_{2p,R\beta_t}(s)=s(R\beta_t)^{\frac{d}{2p}}$,  we have:
\begin{align}
\nonumber
& \mathbb{E}\cro{\exp\pare{\theta\beta_t^{d/q-\alpha}N_{p,R\beta_t}(l_{R\beta_t,\tau})}}
\\
\label{UB2}
& \leq 1+C(\epsilon,\theta,\gamma)\pare{1+\frac{R^{d/(2p)}\sqrt{\beta_t^d G_{R\beta_t,\lambda_t}(0,0)}}{\epsilon\sqrt{2t}}}
e^{\gamma t\epsilon \beta_t^{-\alpha}}
\frac{ \E\cro{\exp\pare{\frac{\gamma (1+\epsilon)}{2}\beta_t^{d/q-\alpha}N_{2p,R\beta_t}^2(Z)}}
^{1/(1+\epsilon)}}{\P\pare{N_{2p,R\beta_t}(Z)\geq \sqrt{2t\epsilon}\beta_t^{-d/(2q)}(1+\sqrt{\epsilon})}} 
\end{align}

Combining \refeq{UB1} and \refeq{UB2}, we have thus proved that for all $\theta > 0, \epsilon > 0, \gamma > 0$ such
that  $\theta<\frac{\gamma}{1+\epsilon}$,
\begin{align}
\nonumber
& \E \cro{\exp\pare{\theta\beta_t^{d/q-\alpha}N_{p}(l_t)}}  e^{-at\beta_t^{-\alpha}} 
\\
 \label{UB3}
& \le
\cro{1+C(\epsilon,\theta,\gamma)\pare{1+\frac{R^{d/(2p)}\sqrt{\beta_t^d G_{R\beta_t,\lambda_t}(0,0)}}{\epsilon\sqrt{2t}}}
e^{\gamma t\epsilon \beta_t^{-\alpha}}
\frac{ \E\cro{\exp\pare{\frac{\gamma (1+\epsilon)}{2}\beta_t^{d/q-\alpha}N_{2p,R\beta_t}^2(Z)}}
^{1/(1+\epsilon)}}{\P\pare{N_{2p,R\beta_t}(Z)\geq \sqrt{2t\epsilon}\beta_t^{-d/(2q)}(1+\sqrt{\epsilon})}} }
.
\end{align}

We now use concentration inequalities for norms of Gaussian processes.

\begin{lemma} Large deviations for $N_{2p,R}(Z)$. \\
\label{Z}
Let $\tau$ and $(Z_x, x\in \T_{R\beta_t})$ be defined
as in Theorem \ref{Eisenbaum}, and
$\rho(a,R,t)$ be defined by \refeq{rhoart}.
Under assumptions of  theorem \ref{OKCAVA3},
\begin{enumerate}
\item For all $a>0$, $R>0$, $t> 0$, $\rho(a,R,t) \le a R^{d/q} \beta_t^{d/q-\alpha}$. 
\item $\forall a,\epsilon,R,t>0$, 
\[ \P\cro{N_{2p,R\beta_t} (Z)\geq \sqrt{t \epsilon}\beta_t^{-d/(2q)}}
\geq
 \frac{\beta_t^{d/(2q)}}{\sqrt{2 \pi t \epsilon \rho(a,R,t)}} 
\pare{1-\frac{\beta_t^{d/q}}{t \epsilon\rho(a,R,t)}}_+ 
\exp\pare{- \frac{1}{2}t\beta_t^{-d/q }\epsilon \rho(a,R, t)} .
\]
\item  $\forall a,\gamma,\epsilon,R,t>0$ such that $\gamma(1+\epsilon)^2<\rho(a,R,t)\beta_t^{\alpha-d/q}$
\[
\E\cro{\exp\pare{\frac{\gamma (1+\epsilon)}{2}\beta_t^{d/q-\alpha}N_{2p,R\beta_t}^2(Z)}}
\leq \frac{2}{\sqrt{1-\frac{\gamma(1+\epsilon)^2\beta_t^{d/q-\alpha}}{\rho(a,R,t)}}} 
  .
\]
\end{enumerate} 
\end{lemma} 

\begin{proof}
\begin{enumerate}
\item It suffices to 
take $f=(R\beta_t)^{-d/2p}$ in \refeq{rhoart} to obtain the result.
\item By H\"older's inequality, for any $f$ such that $N_{(2p)',R\beta_t}(f)=1$,
\[
\P\cro{N_{2p,R\beta_t} (Z)\geq \sqrt{t \epsilon}\beta_t^{-d/(2q)}}
 \geq \P \cro{\sum_{x \in \T_{R\beta_t}} f_x Z_x \geq \sqrt{ t  \epsilon}\beta_t^{-d/(2q)} }
\, .
\]
Since $\sum\limits_{x \in \T_{R\beta_t}} f_x Z_x$ is a real centered Gaussian variable with
variance 
\[ \sigma^2_{a,R,t}(f)= 
\sum_{x,y  \in \T_{R\beta_t}} G_{R\beta_t,\lambda_t}(x,y) f_x f_y \, ,
\]
we have:
\begin{align*} 
\P\cro{\nor{Z}_{2p,R\beta_t} \geq  \sqrt{t \epsilon}\beta_t^{-d/(2q)}}
& \geq 
\frac{\sigma_{a,R,t}(f)\beta_t^{d/(2q)}}{\sqrt{2\pi t \epsilon}} 
\pare{1 - \frac{\sigma^2_{a,R,t}(f)\beta_t^{d/q}}{t \epsilon}}_+ 
\exp\pare{- \frac{t\beta_t^{-d/q} \epsilon }{2 \sigma^2_{a,R,t}(f)}} 
 \\
& \geq 
\frac{\sigma_{a,R,t}(f)\beta_t^{d/(2q)}}{\sqrt{2\pi t\epsilon}} 
\pare{1 - \frac{\rho_1(a,R,t)\beta_t^{d/q}}{t\epsilon}}_+ 
\exp\pare{- \frac{t\beta_t^{-d/q}\epsilon }{2 \sigma^2_{a,R,t}(f)}},
\end{align*}
where $\rho_1(a,R,t)=\sup\acc{\sigma^2_{a,R,t}(f), N_{(2p)',R\beta_t}(f)=1}$.
Taking the supremum over $f$ we obtain that $\forall a,R,t,\epsilon >0$,
\begin{align*}
&\P\cro{N_{2p,R\beta_t} (Z)\geq \sqrt{t \epsilon}\beta_t^{-d/(2q)}}\\
\geq &
 \frac{\sqrt{\rho_1(a,R, t)}\beta_t^{d/(2q)}}{\sqrt{2 \pi t\epsilon }} 
\pare{1-\frac{\rho_1(a,R,t)\beta_t^{d/q}}{t\epsilon}}_+ 
\exp\pare{- \frac{t\beta_t^{-d/q} \epsilon }{2 \rho_1(a,R, t)}}. 
\end{align*}
Then it suffices to prove that $\rho_1(a,R,t)=\frac{1}{\rho(a,R,t)}$ to have the result.

On one hand, by H\"older inequality, 
$$<f,G_{R\beta_t,\lambda_t}f>_{R\beta_t}\leq N_{2p,R\beta_t}(G_{R\beta_t,\lambda_t}f),\ \forall f \text{ such that }N_{(2p)',R\beta_t}(f)=1.$$ 
Since $G_{R\beta_t,\lambda_t}^{-1}=\lambda_t -A_{R\beta_t}$,
\begin{align*}
<f,G_{R\beta_t,\lambda_t}f>_{R\beta_t}&=<G_{R\beta_t,\lambda_t}^{-1}G_{R\beta_t,\lambda_t}f,G_{R\beta_t,\lambda_t}f>_{R\beta_t}\\
&= \lambda_tN_{2,R\beta_t}^2(G_{R\beta_t,\lambda_t}f) - <A_{R\beta_t}G_{R\beta_t,\lambda_t}f,G_{R\beta_t,\lambda_t}f>\\
&\geq \rho(a,R,t)N_{2p,R\beta_t}^2(G_{R\beta_t,\lambda_t}f).
\end{align*}
Therefore, for all $f$ such that $N_{(2p)',R\beta_t}(f)=1$, $<f,G_{R\beta_t,\lambda_t}f>_{R\beta_t}^2\leq \frac{<f,G_{R\beta_t,\lambda_t}f>_{R\beta_t}}{\rho(a,R,t)}$. Then, taking the supremum over $f$, $\rho_1(a,R,t)\leq 1/\rho(a,R,t)$. 

On the other hand, let $f_0$ achieving the infimum in the definition of $\rho(a,R,t)$. 
\begin{align*}
\rho_1(a,R,t)&= \sup_{N_{(2p)',R\beta_t}(f)=1} \acc{<f,G_{R\beta_t,\lambda_t}f>_{R\beta_t}}\\
&\geq \frac{<G_{R\beta_t,\lambda_t}^{-1}f_0,f_0>_{R\beta_t}}{N_{(2p)',R\beta_t}^2(G_{R\beta_t,\lambda_t}^{-1}f_0)}
= \frac{\rho(a,R,t)}{N_{(2p)',R\beta_t}^2(G_{R\beta_t,\lambda_t}^{-1}f_0)}.
\end{align*}
Furthermore, using the Lagrange multipliers method, we know that\\
 $N_{(2p)',R\beta_t}(G_{R\beta_t,\lambda_t}^{-1}f_0)=\rho(a,R,t)$. Hence $\rho_1(a,R,t)\geq 1/\rho(a,R,t)$, and then
$\rho_1(a,R,t)= 1/\rho(a,R,t)$. 

\item
Let $M$ be a median of $N_{2p,R\beta_t}(Z)$. We can easily see that
\begin{align}
\nonumber
&\E\cro{\exp\pare{\frac{\gamma(1+\epsilon)}{2} \beta_t^{d/q-\alpha}N_{2p,R\beta_t}^2(Z) }}
\\
\label{UB4}
\leq & \E\cro{\exp\pare{\frac{\gamma(1+\epsilon)^2}{2} \beta_t^{d/q-\alpha} \vert   N_{2p,R\beta_t}-M\vert^2}} \exp\pare{\frac{\gamma(1+\epsilon)^2}{2\epsilon} \beta_t^{d/q-\alpha}M^2}
\end{align}

Let us now prove that under our assumptions we have $ \beta_t^{d/q-\alpha}M^2=o( t\beta_t^{-\alpha})$ which is equivalent to $M^2=o( t\beta_t^{-d/q})$.
Since $M=\pare{\text{median}\pare{\sum\limits_{x\in\T_{R\beta_t}} Z_x^{2p}}}^{1/2p}$ and that for $X\geq 0,\ \text{ median}(X)\leq 2 \E[X]$, we get:
\begin{align*}
M^2 &=\pare{\text{median}\pare{\sum_{x\in\T_{R\beta_t}} Z_x^{2p}}}^{1/p}
\leq \pare{2\E\cro{\sum_{x\in\T_{R\beta_t}} Z_x^{2p}}}^{1/p}
\\
& \le 2^{1/p} \pare{\sum_{x\in\T_{R\beta_t}} G_{R\beta_t,\lambda_t}(0,0)^p \E\cro{Y^{2p}}}^{1/p},\text{ where } Y\sim\mathcal{N}(0,1)\\
&\leq C(p)(R\beta_t)^{d/p}G_{R\beta_t,\lambda_t}(0,0).
\end{align*}
The asymptotic behavior of $G_{R\beta_t,\lambda_t}(0,0)$ is given by the following lemma whose proof
is postponed in section \ref{limrho}. 
\begin{lemma}{Behavior of $G_{R\alpha_t,\lambda_t}(0,0)$}.\\
\label{G}
Assume {\bf (H4)}, and that $\lambda_t=a\beta_t^{-\alpha}$ and $\beta_t\gg 1$. Then for any $a,R>0$,
\begin{enumerate}
\item for $d<\alpha$, $G_{R\beta_t,\lambda_t}(0,0)=O(\beta_t^{\alpha-d})$.
\item for $d=\alpha$, $G_{R\beta_t,\lambda_t}(0,0)=O(\log\beta_t)$.
\item for $d>\alpha$, $G_{R\beta_t,\lambda_t}(0,0)=O(1)$.
\end{enumerate}
\end{lemma}
Hence, for $d<\alpha$, $M^2 = O (\beta_t^{\alpha-d/q})=o( t\beta_t^{-d/q})$ as soon as $\beta_t^\alpha \ll t $. For $d=\alpha$, $M^2=O(\beta_t^{d/p}\log\beta_t)=o( t\beta_t^{-d/q})$ as soon as $\beta_t^d\ll \frac{t}{\log(t)}$. 
For $d>\alpha$, $M^2 =O( \beta_t^{d/p})=o (t\beta_t^{-d/q})$ as soon as
$\beta_t^d\ll t$.

Let us now work on the expectation in \refeq{UB4}.
\begin{align*}
&\E\cro{\exp\pare{\frac{\gamma(1+\epsilon)^2}{2} \beta_t^{d/q-\alpha} \vert   N_{2p,R\beta_t}-M\vert^2}}\\
= & 1+ \int_0^{+\infty} \frac{\gamma(1+\epsilon)^2}{2}\beta_t^{d/q-\alpha}\exp\pare{\frac{\gamma(1+\epsilon)^2}{2}\beta_t^{d/q-\alpha}y}\P(\vert  N_{2p,R\beta_t}(Z)-M\vert^2>y )dy 
\end{align*}

Using concentration inequalities for norms of Gaussian processes (see for instance lemma 3.1 in \cite{LT91}), for all $y > 0$, 

\[ \P\pare{\va{N_{2p,R\beta_t} (Z)- M} \geq \sqrt{y}} 
\leq 2 \P\pare{ Y \geq \sqrt{y\rho(a, R, t) }} \, , \mbox{ where } Y\sim \mathcal{N}(0,1) \, .
\] 
 Hence,
\begin{align*}
&\E\cro{\exp\pare{\frac{\gamma(1+\epsilon)^2}{2} \beta_t^{d/q-\alpha} \vert   N_{2p,R\beta_t}-M\vert^2}}\\
\leq &1+ 2\int_0^{+\infty} \gamma(1+\epsilon)^2\beta_t^{d/q-\alpha}\exp\pare{\frac{\gamma(1+\epsilon)^2}{2}\beta_t^{d/q-\alpha}y} \P\pare{Y\geq \sqrt{y\rho(a,R,t)}}dy
\\
= & 2 \E \cro{\exp \pare{  \frac{\gamma (1+\epsilon)^2}{2 \rho(a,R,t)} \beta_t^{d/q - \alpha} Y^2}} -1
\\
\leq &\frac{2}{\sqrt{1-\frac{\gamma(1+\epsilon)^2\beta_t^{d/q-\alpha}}{\rho(a,R,t)}}}   , 
\text{ if } \gamma  (1+\epsilon)^2 <\rho(a,R,t) \beta_t^{\alpha-d/q}.
\end{align*}

\end{enumerate}
\end{proof}

We now end the proof of the upper bound in \refeq{loglapNp}. To begin with, we state the following
lemma about the asymptotic behavior of $\rho(a,R,t)$. Its proof will be given in Section \ref{limrho}.
\begin{prop}
\label{OKCAVA4}
For positive real numbers $a, R$, let us define
\begin{equation}
\label{rho(a,R)}
\rho(a,R)=\inf_{\underset{R-periodic}{g:\R^d\rightarrow\R}}\acc{ a\nor{g}_{2,R}^2 +R^d\sum_{z\in\Z^d}\va{\F_R(g)\pare{z} }^2\va{\frac{z}{R}}^\alpha  ,\nor{g}_{2p,R}=1} .
\end{equation}
\begin{equation}
\label{rhoa}
\rho(a)=\inf_{g:\R^d\rightarrow \R}\acc{ a\nor{g}_2^2+ \int\limits_{\R^d}\vert \omega\vert^\alpha \vert\F(g)(\omega)\vert^2, \nor{g}_{2p}=1} .
\end{equation} 
Then, under assumptions of Theorem \ref{OKCAVA3}, 
\[ \liminf_{t\rightarrow +\infty}\beta_t^{\alpha-d/q} \rho(a,R,t) \ge \rho(a,R) \, , \,\, \text{ and } 
 \liminf_{R\rightarrow +\infty} \rho(a,R) \geq \rho(a) .
 \]
 Moreover, $\rho(a)= a^{1-d/(\alpha q)} \rho_{\alpha, d, p}$, where $\rho_{\alpha, d, p}$ is defined by \refeq{rho}. 
\end{prop}

Let us fix $\theta, a, R, \gamma, \epsilon$ such that   
$\theta(1+ \epsilon) < \gamma < (1+\epsilon)^{-2} \rho(a,R) \le \liminf_{t \rightarrow +\infty}\beta_t^{\alpha-d/q} \rho(a,R,t)$.
By lemma \ref{G}, 
\[ \limsup_{t \rightarrow +\infty} \frac{t}{\beta_t^{\alpha}} \log \pare{\frac{\beta_t^d G_{R\beta_t,\lambda_t}(0,0)}{t}}\le 0
\, .
\]
%
By 3 of lemma \ref{Z},
\begin{equation}
\limsup_{t\rightarrow+\infty}\frac{\beta_t^\alpha}{t}\log \E\cro{\exp\pare{\frac{\gamma (1+\epsilon)}{2}\beta_t^{d/q-\alpha}N_{2p,R\beta_t}^2(Z)}}\leq 0,
\end{equation}

and by 1 and 2 of lemma \ref{Z}, 
\begin{equation}
\liminf_{t\rightarrow+\infty}\frac{\beta_t^\alpha}{t}\log \P\pare{N_{2p,R\beta_t}(Z)\geq \sqrt{2t\epsilon}\beta_t^{-d/(2q)}(1+\sqrt{\epsilon})}\geq -a\epsilon(1+\sqrt{\epsilon})^2R^{d/q}.
\end{equation}

Therefore, it follows from \refeq{UB3} that  for $\theta(1+\epsilon)<\gamma <(1+\epsilon)^{-2} \rho(a,R)$,
\begin{equation}
\limsup_{t\rightarrow+\infty}\frac{\beta_t^\alpha}{t}\log  \E\cro{\exp\pare{\theta\beta_t^{d/q-\alpha}N_p(l_t)}} \leq a +\epsilon\gamma +a\epsilon (1+\sqrt{\epsilon})^2R^{d/q} .
\end{equation}

Letting  $\epsilon$ go to  $0$, we obtain that  for $\theta<\gamma <\rho(a,R)$ 
\begin{equation}
\limsup_{t\rightarrow+\infty}  \frac{\beta_t^\alpha}{t}\log\E\cro{\exp\pare{\theta\beta_t^{d/q-\alpha}N_p(l_t)}} \leq a .
\end{equation}
Letting $R$ go to $+\infty$, the same inequality is true for $\theta < \rho(a)$. Thus, for all $\theta > 0$, 
\begin{align*}
   \limsup_{t\rightarrow+\infty}  \frac{\beta_t^\alpha}{t}\log\E\cro{\exp\pare{\theta\beta_t^{d/q-\alpha}N_p(l_t)}} 
& \leq \inf\acc{a,\ \rho(a)>\theta} 
\\
& = \inf\acc{a > 0, a^{1-\frac{d}{\alpha q}} \rho_{\alpha,d,p} > \theta} = 
\pare{\frac{\theta}{\rho_{\alpha,d,p}}}^\frac{\alpha q}{\alpha q - d} \, 
\end{align*}
which is the desired result.

\section{Asymptotic behavior of $\rho(a,R,t)$ and $G_{R\beta_t,\lambda_t}(0,0)$.}
\label{limrho}
The aim of this section is to prove  Proposition \ref{OKCAVA4} and Lemma \ref{G}. 

\subsection{Proof of Lemma \ref{G}}.
Let $p_s$ and $p_s^{R\beta_t}$ the transition probabilities of $X_s$ and $X_s^{R\beta_t}$.
\begin{align*}
G_{R\beta_t,\lambda_t}(0,0)	&=\int_0^{+\infty} \exp(-\lambda_t s)p_s^{R\beta_t}(0,0)ds\\
&\leq 1+\int_1^{+\infty}\exp(-\lambda_t s)\sum_{z\in\Z^d} p_s(0,zR\beta_t)ds.
\end{align*}
By {\bf (H4)}, there exists $C$ such that
$$\forall s> 0,\forall z\in\Z^d, p_s(0,zR\beta_t)\leq C\pare{s^{-d/\alpha}\wedge \frac{s}{\abs{zR\beta_t}^{d+\alpha}}}.$$
Then, 
\begin{align*}
G_{R\beta_t,\lambda_t}(0,0)
&\leq 1+C\int_1^{+\infty} \exp(-\lambda_t s)\pare{s^{-d/\alpha} +\sum_{0<\abs{z}\leq \frac{s^{1/\alpha}}{R\beta_t}}s^{-d/\alpha} + \sum_{\abs{z}> \frac{s^{1/\alpha}}{R\beta_t}} \frac{s}{\abs{zR\beta_t}^{d+\alpha}}}ds\\
&\leq 1 + C\pare{\int_1^{\beta_t^\alpha}s^{-d/\alpha}ds +\int_{\beta_t^\alpha}^{+\infty}\frac{\exp(-\lambda_ts)}{\beta_t^d}ds} + C\int_1^{+\infty} \frac{\exp(-\lambda_ts)}{(R\beta_t)^d}ds\\
&\leq 1+ C\int_1^{\beta_t^\alpha}s^{-d/\alpha}ds +C\frac{\exp(-\lambda_t\beta_t^\alpha)}{\lambda_t\beta_t^d} +C\frac{\exp(-\lambda_t)}{\lambda_t(R\beta_t)^d}. 
\end{align*}
Remember that $\lambda_t=a\beta_t^{-\alpha}$. Then  for a constant $C$ depending on $a$ and $R$, we have,
\begin{align*}
G_{R\beta_t,\lambda_t}(0,0)
&\leq 1+ C\int_1^{\beta_t^\alpha}s^{-d/\alpha}ds + O\pare{\beta_t^{\alpha-d}}.
\end{align*}
This inequality yields the result in the three cases $\alpha<d,\alpha=d$ and $\alpha>d$.

\subsection{Proof of Proposition \ref{OKCAVA4}.} {\bf Asymptotic behavior of $\rho(a,R,t)$ when $t \rightarrow + \infty$.} \\


We want to prove that $\liminf_{t \rightarrow + \infty} \rho(a,R,t) \ge \rho(a,R).$ 
\subsubsection{Expression of $\rho(a,R,t)$ in terms of Fourier transform.}
Let us rewrite $\rho(a,R,t)$ defined by (\ref{rhoart}) in terms of Fourier transform.
Let $h$ be the function achieving the infimum in the definition of $\rho(a,R,t)$.
\begin{align*}
&<h,A_{R\beta_t} h>_{R\beta_t}\\
=& \sum_{x\in\T_{R\beta_t}} h(x) \sum_{y\in\T_{R\beta_t}} \mu_{R\beta_t} (y-x) (h(y)-h(x))\\
=& -N_{2,R\beta_t}^2(h) + \sum_{x\in\T_{R\beta_t}} \sum_{y\in\T_{R\beta_t}}h(x) \mu_{R\beta_t} (y-x) h(y)\\
=& -\frac{1}{(R\beta_t)^d}N_{2,R\beta_t}^2(F_{R\beta_t}\pare{h})\\
& + \sum_{x\in\T_{R\beta_t}} \sum_{y\in\T_{R\beta_t}}h(x)h(y) \sum_{z\in\T_{R\beta_t}} \frac{1}{(R\beta_t)^d} F_{R\beta_t}\pare{\mu_{R\beta_t}}(z)\exp\pare{2i\pi\frac{<z,y-x>}{R\beta_t}}\\
=&-\frac{1}{(R\beta_t)^d}N_{2,R\beta_t}^2(F_{R\beta_t}\pare{h}) +\frac{1}{(R\beta_t)^d} \sum_{z\in\T_{R\beta_t}} \vert F_{R\beta_t}\pare{h}(z)\vert^2F_{R\beta_t}\pare{\mu_{R\beta_t}}(z)\\
=&\frac{1}{(R\beta_t)^d} \sum_{z\in\T_{R\beta_t}} \vert F_{R\beta_t}\pare{h}(z)\vert^2\pare{F\pare{\mu}\pare{\frac{z}{R\beta_t}}-1}.
\end{align*}
We have thus shown that 
\begin{equation}
\label{rho(a,R,t)}
\beta_t^{\alpha-d/q}\rho(a,R,t)=a\beta_t^{-\frac{d}{q}}N_{2,R\beta_t}^2(h)+\frac{\beta_t^{\alpha-\frac{d}{q}}}{(R\beta_t)^d}\sum_{z\in\T_{R\beta_t}} \vert F_{R\beta_t}\pare{h}(z)\vert^2(1-F\pare{\mu}(\frac{z}{R\beta_t})).
\end{equation}
We call $a\beta_t^{-\frac{d}{q}}N_{2,R\beta_t}^2(h)$ the $2$-norm part and the second term in the right-hand side 
of \refeq{rho(a,R,t)} the gradient part. 
We set for $x \in \R^d$, 
\[ g_R(x)= \beta_t^{d/(2p)}\sum_{k\in\T_{R\beta_t}}h(k)\varphi(\beta_t x-k) \, ,
\]
where $\varphi$ is the $R\beta_t$-periodic function from $\R^d$ to $\C$, whose Fourier's coefficients are 
 \[ \F_{R\beta_t}(\varphi)(n)=\begin{cases}
 		\frac{1}{(R\beta_t)^d}
 \frac{\sqrt{1-F\pare{\mu}(\frac{n}{R\beta_t})}}{\vert\frac{n}{R\beta_t}\vert^{\alpha/2}}\ind_{[\![-\frac{R\beta_t}{2},\frac{R\beta_t}{2}[\![^d
}(n) 
 & \text{ if } n \ne 0  , \\
\frac{1}{(R\beta_t)^d} & \text{ otherwise, }
\end{cases} \]
where for any $a,b\in\R$, $[\![a,b]\!]^d=[a,b]^d\bigcap \Z^d$.

Note that $\mu$ being a symmetric probability measure, $F(\mu)$ is real valued and $\va{F(\mu)} \le 1$, so 
that $\F_{R\beta_t}(\varphi)(n)$ is well-defined. By definition, $g_R$ is a R-periodic function defined on $\R^d$ and 
for all $z \in \Z^d$,
\[\FF_R(g_R)(z) = \beta_t^{d/2p} F_{R\beta_t}(h)(z)  \FF_{R \beta_t}(\varphi)(z) \, . 
\]

$g_R$ is our candidate to achieve the infimum in the definition \refeq{rho(a,R)} of $\rho(a,R)$.

\subsubsection{Gradient part}
The function $g_R$ was built to preserve the gradient part. Indeed,

\begin{align}
\nonumber
& \frac{\beta_t^{\alpha-\frac{d}{q}}}{(R\beta_t)^d}\sum_{z\in\T_{R\beta_t}} \abs{F_{R\beta_t}\pare{h}(z)}^2(1-F\pare{\mu}(\frac{z}{R\beta_t}))
\\
\nonumber
& \hspace*{2cm} = \frac{\beta_t^{\alpha-\frac{d}{q}}}{(R\beta_t)^d}\sum_{z\in[\![-\frac{R\beta_t}{2},\frac{R\beta_t}{2}[\![^d} \abs{ F_{R\beta_t}\pare{h}(z)}^2 \F_{R\beta_t}(\varphi)(z)^2(R\beta_t)^{2d}\abs{\frac{z}{R\beta_t}}^{\alpha}
\\
\label{grad}
& \hspace*{2cm} = R^d\sum_{z\in\Z^d} \abs{ \F_R(g_R)(z)}^2 \abs{ \frac{z}{R}}^\alpha.
\end{align}

\subsubsection{$2$-norm part}
We work now on the $2$-norm part. By Parseval's equality,

\begin{align*}
\nor{g_R}_{2,R}^2
=&R^d N_{2}^2(\F_R(g_R))
= R^d\sum_{z\in [\![-\frac{R\beta_t}{2},\frac{R\beta_t}{2}[\![^d}\beta_t^{d/p}\va{ F_{R\beta_t}\pare{h}(z)}^2
\va{\F_{R \beta_t}(\varphi)(\frac{z}{R\beta_t})}^2
\\
= &\frac{R^d\beta_t^{d/p}}{(R\beta_t)^{2d}}N_{2,R\beta_t}^2(F_{R\beta_t}\pare{h})  
+\frac{R^d\beta_t^{d/p}}{(R\beta_t)^{2d}}\sum_{z\in[\![-\frac{R\beta_t}{2},\frac{R\beta_t}{2}[\![^d,\ z\ne0} \abs{F_{R\beta_t}\pare{h}(z)}^2
\cro{\frac{1-F\pare{\mu}(\frac{z}{R\beta_t})}{\va{\frac{z}{R\beta_t}}^\alpha}-1}
\\
=& \beta_t^{-d/q} N_{R\beta_t}^2(h)+\frac{R^d\beta_t^{d/p}}{(R\beta_t)^{2d}}\sum_{z\in[\![-\frac{R\beta_t}{2},\frac{R\beta_t}{2}[\![^d,\ z\ne0} \abs{F_{R\beta_t}\pare{h}(z)}^2
\cro{\frac{1-F\pare{\mu}(\frac{z}{R\beta_t})}{\va{\frac{z}{R\beta_t}}^\alpha}-1} .
\end{align*}
Under assumptions {\bf (H1)} and {\bf (H2)}
$F\pare{\mu}(u) \underset{u\rightarrow 0}{=}1-\abs{u}^\alpha+o(\abs{u}^\alpha)$ (see for instance Feller \cite{feller}, Chapter XVI.5).  So, for  any $\epsilon>0$, one can find $\delta>0$ such that for all $\va{z}<\delta$, we have 
$\left\lvert \frac{1-F(\mu)(z)}{\va{z}^\alpha}-1\right\rvert \leq \epsilon$.
Then, 
\begin{align*}
&\frac{R^d\beta_t^{d/p}}{(R\beta_t)^{2d}}\sum_{z\in[\![-\frac{R\beta_t}{2},\frac{R\beta_t}{2}[\![^d,\ z \ne 0}\va{ F_{R\beta_t}\pare{h}(z)}^2
\cro{\frac{1-F\pare{\mu}(\frac{z}{R\beta_t})}{\va{\frac{z}{R\beta_t}}^\alpha}-1}\\
&\leq \frac{R^d\beta_t^{d/p}}{(R\beta_t)^{2d}}\pare{\epsilon\sum_{z/\left\lvert \frac{z}{R\beta_t}\right\rvert<\delta } \va{F_{R\beta_t}\pare{h}(z)}^2   +  \delta^{-\alpha}\sum_{z/\delta\leq\left\lvert \frac{z}{R\beta_t}\right\rvert\leq\frac{1}{2}} \va{F_{R\beta_t}\pare{h}(z)}^2\va{1-F\pare{\mu}\pare{\frac{z}{R\beta_t}}}}\\
&\leq \epsilon \beta_t^{-d/q} N_{2,R\beta_t}^2(h) + \delta^{-\alpha}R^d\beta_t^{d/p}\sum_{z\in\Z^d}\abs{F_{R\beta_t}\pare{h}(z) \F_{R\beta_t}(\varphi)(z)}^2\abs{\frac{z}{R\beta_t}}^\alpha\\
& = \epsilon \beta_t^{-d/q} N_{2,R\beta_t}^2(h) + \delta^{-\alpha}R^d\beta_t^{-\alpha} \sum_{z\in\Z^d} \abs{\F_R(g_R)(z) }^2\abs{\frac{z}{R}}^\alpha.
\end{align*}

Moreover, by (\ref{grad}) and (\ref{rho(a,R,t)}) it is easy to see that
\begin{equation}
R^d \sum_{z\in\Z^d} \abs{\F_R(g_R)(z) }^2\abs{\frac{z}{R}}^\alpha
\leq \beta_t^{\alpha-d/q}\rho(a,R,t).
\end{equation}
Hence,
$\forall \epsilon>0$, $\exists\delta>0$ such that 
\begin{equation}
\label{norm2g}
\nor{g_R}_{2,R}^2
\leq (1+\epsilon)\beta_t^{-d/q} N_{2,R\beta_t}^2(h)+(\delta \beta_t)^{-\alpha}  \pare{\beta_t^{\alpha-d/q}\rho(a,R,t)}.\end{equation}

\subsubsection{$2p$-norm part}
We work now on the most difficult part, that is the $2p$-norm. Using Fourier inversion formula, 
\[
g_R(x)
=\beta_t^{d/(2p)}\sum_{n\in\Z^d} F_{R\beta_t}\pare{h}(n)\F_{R\beta_t}(\varphi)(n)\exp\pare{2i\pi\frac{<x,n>}{R}} \, .
\]
\begin{eqnarray}
\nonumber
\nor{g_R}_{2p,R}
& \geq & \frac{\beta_t^{d/(2p)}}{(R\beta_t)^{d}}\nor{\sum_{n\in[\![-\frac{R\beta_t}{2},\frac{R\beta_t}{2}[\![^d} F_{R\beta_t}\pare{h}(n)\exp\pare{2i\pi\frac{<\cdot,n>}{R}}}_{2p,R}  \\
&& -   \beta_t^{d/(2p)}\nor{\sum_{n\in[\![-\frac{R\beta_t}{2},\frac{R\beta_t}{2}[\![^d} F_{R\beta_t}\pare{h}(n)\pare{\F_{R\beta_t}(\varphi)(n)-\frac{1}{(R\beta_t)^d}}\exp\pare{2i\pi\frac{<\cdot,n>}{R}}}_{2p,R}
\label{2pnorm1}
\end{eqnarray}

We want to prove that the first term is close to the $2p$-norm of $h$, and that the second term is negligible.
We first work on the second term in (\ref{2pnorm1}). Let $\delta_t\underset{t\rightarrow +\infty}{\longrightarrow }0 $ to be chosen later. We first use inversion formula of Proposition \ref{inversion périodique}, Young inequality of Proposition \ref{inégalité Fourier} and  we again cut the sum in two parts. Denoting by $(2p)'$ the conjugate exponent of $2p>2$, we are led to

\begin{align}
\nonumber&\beta_t^{d/(2p)}\nor{\sum_{n\in[\![-\frac{R\beta_t}{2},\frac{R\beta_t}{2}[\![^d} F_{R\beta_t}\pare{h}(n)\pare{\F_{R\beta_t}(\varphi)(n)-\frac{1}{(R\beta_t)^d}}\exp\pare{2i\pi\frac{<\cdot,n>}{R}}}_{2p,R}\\
\nonumber\leq&(R\beta_t)^{d/(2p)} N_{(2p)',R\beta_t}\pare{ F_{R\beta_t}\pare{h}(\cdot) 
	\ind_{[\![-\frac{R\beta_t}{2},\frac{R\beta_t}{2}[\![^d}(\cdot) 
	\pare{\F_{R\beta_t}(\varphi)(\cdot)-\frac{1}{(R\beta_t)^d}}} \leq  T_1 + T_2 , 
\end{align}
with $\begin{array}[t]{ll}
T_1 & =(R\beta_t)^{d/(2p)}\pare{\sum_{n/\abs{\frac{n}{R\beta_t}}\leq\delta_t} \abs{F_{R\beta_t}\pare{h}(n)\pare{\F_{R\beta_t}(\varphi)(n)-\frac{1}{(R\beta_t)^d}}}^{(2p)'}}^{1/(2p)'}, 
\\
T_2 & = (R\beta_t)^{d/(2p)}\pare{\sum_{n/\frac{1}{2}\geq\abs{\frac{n}{R\beta_t}}>\delta_t} \abs{F_{R\beta_t}\pare{h}(n)\pare{\F_{R\beta_t}(\varphi)(n)-\frac{1}{(R\beta_t)^d}}}^{(2p)'}}^{1/(2p)'}.
\end{array}$

Let us focus on the first term $T_1$.  Set $\epsilon(u)=\va{u}^{-\alpha/2} \sqrt{1-F(\mu)(u)} - 1$. 
By our assumptions on $\mu$, $\lim_{u \rightarrow 0} \epsilon(u)=0$. Using H\"older inequality ($(2p)' < 2$) and Parseval's equality,
\begin{align*} 
T_1 \leq& \frac{(R\beta_t)^{d/(2p)}}{(R\beta_t)^d}\pare{\sum_{n/0<\abs{\frac{n}{R\beta_t}}\leq\delta_t} \abs{F_{R\beta_t}\pare{h}(n)\pare{\frac{\sqrt{1-F\pare{\mu}(\frac{n}{R\beta_t})}}{\va{\frac{n}{R\beta_t}}^{\alpha/2}}-1}}^{(2p)'}}^{1/(2p)'}
\\
\leq &\frac{(R\beta_t)^{d/(2p)}}{(R\beta_t)^d} N_{2,R\beta_t}(F_{R\beta_t}\pare{h})  \pare{\sum_{n/0<\abs{\frac{n}{R\beta_t}}\leq\delta_t} \abs{\frac{\sqrt{1-F\pare{\mu}(\frac{n}{R\beta_t})}}{\va{\frac{n}{R\beta_t}}^{\alpha/2}}-1}^{2q}}^{1/(2q)}
\\
=& \frac{(R\beta_t)^{d/(2p)}}{(R\beta_t)^{d/2}}N_{2,R\beta_t}(h)\pare{\sum_{n/\abs{\frac{n}{R\beta_t}}\leq\delta_t} \abs{\epsilon\pare{\frac{n}{R\beta_t}}}^{2q}}^{1/(2q)} \\
\leq&\frac{1}{\sqrt{a}} \pare{(\beta_t\delta_t)^{d/(2q)}\sup_{\abs{u}\leq \delta_t}\va{\epsilon(u)}}
\sqrt{\beta_t^{\alpha-d/q} \rho(a,R,t) } ,
\end{align*}
 by \refeq{rho(a,R,t)}. To be negligible compared to $\sqrt{\beta_t^{\alpha-d/q} \rho(a,R,t)}$, we have to choose $\delta_t$ such that 
\begin{equation}
\label{delta1}
(\beta_t\delta_t)^{d/(2q)}\sup_{\abs{u}\leq \delta_t} \va{\epsilon(u)}\underset{t\rightarrow +\infty}{\rightarrow} 0.
\end{equation}

Let us turn to the second term $T_2$. By assumption {\bf (H3)}, there exists $C>0$ such
 that for $\va{x} \leq \frac{1}{2}$, $\va{x}^{\alpha/2}\leq C\sqrt{1-F\pare{\mu}(x)}$. Then, using the fact that 
 $\abs{\frac{n}{R\beta_t}} \leq \frac{1}{2}$ and H\"older inequality, 
\begin{align*}
T_2 
\leq& C\frac{(R\beta_t)^{d/(2p)}}{(R\beta_t)^d}\pare{\sum_{n/\frac{1}{2}\geq\abs{\frac{n}{R\beta_t}}>\delta_t} \abs{F_{R\beta_t}\pare{h}(n)   \frac{\sqrt{1-F\pare{\mu}(\frac{n}{R\beta_t})}}{\vert\frac{n}{R\beta_t}\vert^{\alpha/2}} }^{(2p)'}}^{1/(2p)'} \\
\leq & C\frac{(R\beta_t)^{d/(2p)}}{(R\beta_t)^d} N_{2,R\beta_t}\pare{F_{R\beta_t}\pare{h}(\cdot)\sqrt{1-F\pare{\mu}\pare{\frac{\cdot}{R\beta_t}}}} (R\beta_t)^{\alpha/2} \pare{\sum_{n/\abs{n}>R\beta_t\delta_t}\abs{n}^{-\alpha q}}^{1/(2q)}\\
\leq & C(R\beta_t)^{\alpha/2-d/(2q)-d/2} N_{2,R\beta_t}\pare{F_{R\beta_t}\pare{h}(\cdot)\sqrt{1-F\pare{\mu}\pare{\frac{\cdot}{R\beta_t}}}} \pare{(R\beta_t\delta_t)^{d-\alpha q}}^{1/(2q)}.
 \end{align*}
Using (\ref{rho(a,R,t)}), we have therefore
\[ T_2 
\leq  C\delta_t^{\frac{1}{2}(d/q-\alpha)} \sqrt{\rho(a,R,t)}.
\]

To be negligible compared to $\sqrt{\beta_t^{\alpha-d/q} \rho(a,R,t)}$, it is enough to choose $\delta_t$ such that
\begin{equation}
\label{delta2}
\beta_t\delta_t\rightarrow +\infty.
 \end{equation}
Note that conditions (\ref{delta1}) and (\ref{delta2}) are compatible. Indeed, set $\eta(x) = x \sup_{\va{u} \le x} \va{\epsilon(u)}^{2q/d}$
for $x \in \R^+$. $\eta$ is an increasing function, and  (\ref{delta1}) and (\ref{delta2}) are equivalent to 
$\beta_t^{-1} \ll \delta_t \ll \eta^{-1}(\beta_t^{-1})$. Such a $\delta_t$ can be found as soon as $\lim_{x \rightarrow 0} 
\eta^{-1}(x)/x = + \infty$, which is the case since  $\lim_{x \rightarrow 0} x/\eta(x) = + \infty$.
 We have now succeeded to control the second term in (\ref{2pnorm1}) by proving that one can found $u_t$, $\lim_{t \rightarrow
  + \infty}
 u_t = 0$ such that

 \begin{equation}
 \label{2pnorm2}
 \beta_t^{d/(2p)}\nor{\sum_{n\in[\![-\frac{R\beta_t}{2},\frac{R\beta_t}{2}[\![^d} F_{R\beta_t}\pare{h}(n)\pare{\F_{R\beta_t}(\varphi)(n)-\frac{1}{(R\beta_t)^d}}
 e^{2i\pi\frac{<\cdot,n>}{R}}}_{2p,R}
\le u_t  \sqrt{\beta_t^{\alpha-d/q}\rho(a,R,t)} \, .
 \end{equation}

It remains to control the first term of the sum in (\ref{2pnorm1}). Performing the change of variable $x\rightarrow x\beta_t$,
\[ \frac{\beta_t^{d/(2p)}}{(R\beta_t)^{d}}\nor{\sum_{n\in[\![-\frac{R\beta_t}{2},\frac{R\beta_t}{2}[\![^d} F_{R\beta_t}\pare{h}(n) e^{2i\pi\frac{<\cdot,n>}{R}}}_{2p,R} 
=\nor{ \frac{1}{(R\beta_t)^{d}}\sum_{n\in[\![-\frac{R\beta_t}{2},\frac{R\beta_t}{2}[\![^d} F_{R\beta_t}\pare{h}(n) e^{2i\pi\frac{<\cdot,n>}{R\beta_t}}}_{2p,R\beta_t} .
\]
For $x\in\R^d$, let
 \begin{equation*}
 f(x):= \frac{1}{(R\beta_t)^d}\sum_{n\in[\![-\frac{R\beta_t}{2},\frac{R\beta_t}{2}[\![^d} F_{R\beta_t}\pare{h}(n)\exp\pare{2i\pi\frac{<x,n>}{R\beta_t}}.
 \end{equation*}
 Note that $f(x)=h(x)$, $\forall x\in \T_{R\beta_t}$.
  We define an approximation of $f$ by:
$$\forall x\in \R^d, \bar{f}(x)=\sum_{k\in[\![-\frac{R\beta_t}{2},\frac{R\beta_t}{2}[\![^d} \bar{f}_k\ind_{Q_k}(x)$$
where $Q_k=[k-\frac{1}{2};k+\frac{1}{2}[^d$ is the unit cube centered on $k$,
and $\bar{f}_k=\frac{1}{\abs{Q_k}}\int_{Q_k}f(x)dx$ is
 the mean of $f$ on  $Q_k.$
 For any $x\in\R^d$, we denote by $\lfloor x\rfloor$ the unique $k\in\Z^d$ such that $x\in Q_k$.
 Note that $\forall x\in Q_k, \lfloor x\rfloor=k$ and $ f(\lfloor x\rfloor)=h(k)$.
 Introducing the $R\beta_t$-periodic dunctions, $f(\lfloor\cdot\rfloor)$ and $\bar{f}(\cdot)$, we have that:

\begin{align}
\label{2pnormf}\nor{f}_{2p,R\beta_t}\geq \nor{f(\lfloor\cdot\rfloor)}_{2p,R\beta_t} 
-\nor{\bar{f}-f(\lfloor\cdot\rfloor)}_{2p,R\beta_t}
-\nor{f-\bar{f}}_{2p,R\beta_t}.
\end{align}

The first term is exactly the $2p$-norm of $h$. Let us consider the second one. 
For $x\in \R^d$,
\begin{align*}
\bar{f}(x)-f(\floor{x}) 
=& \frac{1}{(R\beta_t)^d}\sum_{n\in[\![-\frac{R\beta_t}{2},\frac{R\beta_t}{2}[\![^d}F_{R\beta_t}\pare{h}(n)
\underbrace{\int_{Q_0} e^{2i\pi\frac{<n,y>}{R\beta_t}}dy}_{C\pare{\frac{n}{R\beta_t}}}
\sum_{k\in[\![-\frac{R\beta_t}{2},\frac{R\beta_t}{2}[\![^d} e^{2i\pi\frac{<n,k>}{R\beta_t}}\ind_{Q_k}(x) \\
&\hspace*{1cm} 
-\frac{1}{(R\beta_t)^d}\sum_{n\in[\![-\frac{R\beta_t}{2},\frac{R\beta_t}{2}[\![^d}F_{R\beta_t}\pare{h}(n) e^{2i\pi\frac{< n,\lfloor x \rfloor>}{R\beta_t}}\\
=&\frac{1}{(R\beta_t)^d}\sum_{n\in[\![-\frac{R\beta_t}{2},\frac{R\beta_t}{2}[\![^d}F_{R\beta_t}\pare{h}(n)\pare{C\pare{\frac{n}{R\beta_t}}-1}\exp\pare{2i\pi\frac{<n,\lfloor x\rfloor>}{R\beta_t}}
\end{align*}
Therefore, using inversion formula of Proposition \ref{inversion finie}, Young inequality of Proposition \ref{inégalité Fourier} and H\"older inequality:
\begin{align*}
&\nor{\bar{f}-f(\lfloor\cdot\rfloor)}_{2p,R\beta_t}
\\
\leq&(R\beta_t)^{-d/(2p)'}N_{(2p)'}\pare{F_{R\beta_t}\pare{h}(\cdot) 
\ind_{[\![-\frac{R\beta_t}{2},\frac{R\beta_t}{2}[\![^d}(\cdot) 
\pare{C\pare{\frac{\cdot}{R\beta_t}}-1}}
\\
\leq& (R\beta_t)^{-d/(2p)'} N_{2,R\beta_t}\pare{F_{R\beta_t}\pare{h}(\cdot)\sqrt{1-F\pare{\mu}\pare{\frac{\cdot}{R\beta_t}}}}
N_{2q}\pare{\ind_{[\![-\frac{R\beta_t}{2},\frac{R\beta_t}{2}[\![^d}(\cdot) \frac{C\pare{\frac{\cdot}{R\beta_t}}-1}{\sqrt{1-F\pare{\mu}\pare{\frac{\cdot}{R\beta_t}}}}}.
\end{align*}
Let us have a look at  $N_{2q}\pare{\ind_{[\![-\frac{R\beta_t}{2},\frac{R\beta_t}{2}[\![^d}(\cdot)\frac{C\pare{\frac{\cdot}{R\beta_t}}-1}{\sqrt{1-F\pare{\mu}\pare{\frac{\cdot}{R\beta_t}}}}}$. 
A simple computation shows that
$$C(x)=\prod_{j=1}^d e^{i\pi x_j} \frac{\sin(\pi x_j)}{\pi x_j}\, , 
\, \, 
C(x)\underset{0}{\sim} 1+O(\va{x}).$$
Combining this with the fact that $1-F\pare{\mu}(x)\underset{0}{\sim}\va{x}^\alpha$, we have that
$\frac{C(x)-1}{\sqrt{1-F\pare{\mu}(x)}}=O(\va{x}^{1-\alpha/2}).$
Moreover, by {\bf (H3)},  $F\pare{\mu}(x)=1\Leftrightarrow x \in \Z^d$. Then the function $x\rightarrow\frac{C(x)-1}{\sqrt{1-F\pare{\mu}(x)}}$ is uniformly bounded for $\va{x} \le  \frac 1 2$ and 
 $N_{2q}\pare{\ind_{[\![-\frac{R\beta_t}{2},\frac{R\beta_t}{2}[\![^d}(\cdot)\frac{C\pare{\frac{\cdot}{R\beta_t}}-1}{\sqrt{1-F\pare{\mu}\pare{\frac{\cdot}{R\beta_t}}}}}=O((R\beta_t)^{d/(2q)}).$
 Therefore, using (\ref{rho(a,R,t)}),
 \begin{equation}
 \label{fbar-f}
  \nor{\bar{f}-f(\lfloor\cdot\rfloor)}_{2p,R\beta_t}
 \le C \sqrt{\beta_t^{-\alpha + d/q}} \sqrt{\beta_t^{\alpha - d/q}\rho(a,R,t)} .
  \end{equation}
We turn now to the third term of (\ref{2pnormf}). By Poincar\'e's inequality, as the unit cube is a Lypschitz's domain,
\begin{eqnarray*}
\nor{f-\bar{f}}_{2p,R\beta_t}^{2p}
& = & \int_{[-R\beta_t/2; R \beta_t/2[^d} \abs{f(x)-\bar{f}(x)}^{2p} \, dx 
= \sum_{k\in [\![-\frac{R\beta_t}{2},\frac{R\beta_t}{2}[\![^d}\int_{Q_k}\abs{f(x)-\bar{f}_k}^{2p}dx
\\
& \leq & 
C\sum_{k\in [\![-\frac{R\beta_t}{2},\frac{R\beta_t}{2}[\![^d} \nor{\nabla f}_{2p,Q_k}^{2p} = C\nor{\nabla f}_{2p,R\beta_t}^{2p}, 
\end{eqnarray*}
 where C depends only on $d,p$ and $\va{Q_0}$. Hence, 
 denoting by $\psi_j$ the j-th coordinate function $\psi_j(n)=n_j$,  and  using inversion formula of Proposition \ref{inversion périodique}, Young's and H\"older's inequalities,  
\begin{align*}
\nor{f-\bar{f}}_{2p,R\beta_t}
& \leq \frac{C}{(R\beta_t)^d}\sum_{j=1}^d \nor{\sum_{n\in[\![-\frac{R\beta_t}{2},\frac{R\beta_t}{2}[\![^d}F_{R\beta_t}\pare{h}(n)\frac{2i\pi \psi_j(n)}{R\beta_t}\exp\pare{2i\pi\frac{<n,\cdot>}{R\beta_t}}}_{2p,R\beta_t}
\\
& \leq  C\frac{(R\beta_t)^{d/(2p)}}{(R\beta_t)^d} \sum_{j=1}^d N_{(2p)'}\pare{F_{R\beta_t}\pare{h}(\cdot)\frac{2i\psi_j(\cdot)
		}{R\beta_t} \ind_{[\![-\frac{R\beta_t}{2},\frac{R\beta_t}{2}[\![^d}(\cdot)}
\\
& \leq C\frac{(R\beta_t)^{d/(2p)}}{(R\beta_t)^d} N_{(2p)'}\pare{F_{R\beta_t}\pare{h}(\cdot)\frac{\abs{\cdot}}{R\beta_t}
		\ind_{[\![-\frac{R\beta_t}{2},\frac{R\beta_t}{2}[\![^d}(\cdot)}
\\
& \leq  C\frac{(R\beta_t)^{d/(2p)}}{(R\beta_t)^d} N_{2,R\beta_t}\pare{F_{R\beta_t}\pare{h}(\cdot)\sqrt{1-F\pare{\mu}\pare{\frac{\cdot}{R\beta_t}}}}
\\
&\hspace{2.5cm}
N_{2q}\pare{\frac{\va{\cdot}}{R\beta_t}\pare{1-F\pare{\mu}\pare{\frac{\cdot}{R\beta_t}}}^{-1/2} 
			\ind_{[\![-\frac{R\beta_t}{2},\frac{R\beta_t}{2}[\![^d}(\cdot)}
\end{align*}
As previously, $1-F\pare{\mu}(x)\underset{0}{\sim}\va{x}^\alpha$ implies that $\frac{\abs{x}}{\sqrt{1-F\pare{\mu}(x)}}\underset{0}{\sim}\va{x}^{1-\alpha/2}$, which is bounded around $0$ since $\alpha\in]0,2]$. By {\bf (H3)}, 
we deduce that  the function $x\rightarrow\frac{\abs{x}}{\sqrt{1-F\pare{\mu}(x)}} $ is uniformly bounded for $\va{x} \le 1/2$. Therefore,
$
N_{2q}\pare{\frac{\abs{\cdot}}{R\beta_t}\pare{1-F\pare{\mu}\pare{\frac{\cdot}{R\beta_t}}}^{-1/2}
\ind_{[\![-\frac{R\beta_t}{2},\frac{R\beta_t}{2}[\![^d}(\cdot)}
=O((R\beta_t)^{d/(2q)})$, 
and using (\ref{rho(a,R,t)}),
\begin{equation}
\label{f-fbar}
\nor{f-\bar{f}}_{2p,R\beta_t}
\le C \sqrt{\beta_t^{-\alpha + d/q}} \sqrt{\beta_t^{\alpha - d/q}\rho(a,R,t)} .
\end{equation}

Combining (\ref{2pnormf}), (\ref{fbar-f}), (\ref{f-fbar}) and the fact that $N_{2p,R\beta_t}(h)=1$, we have shown that
\begin{align}
\nor{f}_{2p,R\beta_t}&\geq 1- C \sqrt{\beta_t^{-\alpha + d/q}} \sqrt{\beta_t^{\alpha - d/q}\rho(a,R,t)} . 
\label{2pnormf2}
\end{align}

Putting(\ref{2pnorm1}), (\ref{2pnorm2}), (\ref{2pnormf2})  together ,  as $\alpha-d/q>0$, we have controlled the $2p$-norm of $g_R$ by:
\begin{equation}
\nor{g_R}_{2p,R} 
\geq  1 - u_t \sqrt{\beta_t^{\alpha-d/q}\rho(a,R,t)} ,
\label{controle norme 2p}
\end{equation}
for some positive function $(u_t, t \ge 0)$ satisfying $\lim_{t \rightarrow + \infty} u_t =0$.
Finally, $\forall R>0$, $\forall \epsilon>0,\exists \delta>0$ such that
\begin{align*}
& \beta_t^{\alpha-d/q}\rho(a,R,t)
\\
& \hspace*{1cm}
= a\beta_t^{-\frac{d}{q}} N_{2,R\beta_t}^2(h)  -\beta_t^{\alpha-\frac{d}{q}} <h,A_{R\beta_t} h>_{R\beta_t}
\\
& \hspace*{1cm}
\geq \frac{1}{1+\epsilon}\pare{a \nor{g_R}_{2,R}^2+R^d\sum_{z\in\T_{R\beta_t}} \abs{ \F_R(g_R)(z)}^2 \abs{ \frac{z}{R}}^\alpha}
- \frac{a}{1+\epsilon} (\delta \beta_t)^{-\alpha} \pare{\beta_t^{\alpha-d/q}\rho(a,R,t)}
\\
&\hspace*{1cm} 
\geq \frac{\nor{g_R}_{2p,R}^2}{1+\epsilon} \rho(a,R) - \frac{a}{1+\epsilon} (\delta \beta_t)^{-\alpha} \pare{\beta_t^{\alpha-d/q}\rho(a,R,t)}
\\
 &\hspace*{1cm} 
  \geq  \frac{1-u_t \sqrt{\beta_t^{\alpha-d/q}\rho(a,R,t)}}{1+\epsilon}
 \rho(a,R) - \frac{a}{1+\epsilon} (\delta \beta_t)^{-\alpha} \pare{\beta_t^{\alpha-d/q}\rho(a,R,t)} .
\end{align*}
 We can assume that $\liminf_{t \rightarrow + \infty}  \beta_t^{\alpha-d/q}\rho(a,R,t) < + \infty$. Otherwise, there is nothing to
 prove. Letting first $t  \rightarrow + \infty $, then  $\epsilon\rightarrow 0$ in the above inequality, we obtain that for all $R>0$,  
\begin{equation*}
\liminf_{t\rightarrow +\infty}
\beta_t^{\alpha-d/q}\rho(a,R,t) 
\geq \rho(a,R).
\end{equation*}

\subsection{Proof of Proposition \ref{OKCAVA4}.}{\bf Asymptotic behavior of $\rho(a,R)$ when $R \rightarrow + \infty$.}

It remains now to prove that for any $a>0$,
\begin{eqnarray}
I:= \liminf_{R\rightarrow+\infty} \, \rho(a,R) 
\label{lim.per} \geq  \rho(a)
\, . 
\end{eqnarray}
This statement was already proved in \cite{BK,ClemEJP} for $\alpha =2$. We show it for $\alpha < 2$. 
 
We assume that $I<+\infty$, otherwise there is nothing to prove. Let then $\varepsilon>0$. For any $R_0 > 0$, let  $R\geq R_0$ and 
 let $g_R$  be such that $\nor{g_R}_{2p,R}=1$ and 
\begin{equation}\label{eq:agR}
a\nor{g_R}_{2,R}^2+R^d\sum_{z\in\Z^d} \abs{\frac{z}{R}}^\alpha \abs{\mathcal{F}_R(g_R)(z)}^2\leq I+\varepsilon 
\end{equation}
Note that any translation of order $\theta$ of $g_R$ also satisfies (\ref{eq:agR}). Indeed, 
setting $g_{R,\theta}= g_R(\theta+ .)$, and using that $g_R$ is periodic, we get 
 $\nor{g_{R,\theta}}_{2p,R}=\nor{g_R}_{2p,R}$, 
$\nor{g_{R,\theta}}_{2,R}=\nor{g_R}_{2,R}$, 
$\F_R(g_R(\theta+.))(n)=\exp\pare{2i\pi\frac{<n,\theta>}{R}} \F_R(g_R)(n)$, and 
 $\sum_{z\in\Z^d} \abs{\frac{z}{R}}^\alpha \abs{\mathcal{F}_R(g_R)(z)}^2=\sum_{z\in\Z^d} \abs{\frac{z}{R}}^\alpha \abs{\mathcal{F}_R(g_{R,\theta})(z)}^2 $.

Let $E_R = \pare{[0;\sqrt{R}] \cup [R-\sqrt{R};R]}^d$ and $Q_R = [0,R]^d$. 
\[ \inf_{\theta \in Q_R} \int_{E_R} \va{g_{R,\theta}(x)}^{2p} \, dx 
\le \frac{1}{R^d} \int_{Q_R}  \int_{E_R} \va{g_{R,\theta}(x)}^{2p} \, dx  \, d\theta 
= (2 R^{-1/2})^d \, .
\]
Therefore one can found $\theta$ such that $\int_{E_R} \va{g_{R,\theta}(x)}^{2p} \, dx \le 3^d R^{-d/2}$, and
 we can assume without loss of generality that $g_R$ also satisfies 
\begin{equation}
\label{cont.bord}
\int_{E_R} \va{g_{R}(x)}^{2p} \, dx \le 3^d R^{-d/2} \, .
\end{equation}
Let $\phi_R$ be the function from $\R$ to $\R$, whith compact support in $[0;R]$, equal to $1$ on 
$[0,R]\setminus([0;\sqrt{R}] \cup [R-\sqrt{R};R])$, and which is linear on $[0;\sqrt{R}]$ and  $[R-\sqrt{R};R]$.
Let $\psi_R: \R^d \mapsto [0,1]$ be defined by 
$\psi_R(x)=\prod_{i=1}^d \phi_R(x_i)$. Note that $\psi_R$ has compact
support in $Q_R$, is equal to $1$ on $Q_R \setminus E_R$, and that for all $x \in \R^d$, 
$\nor{\nabla \psi_R(x)} \le \sqrt{d/R}$.

Let take $g(x)=g_{R}(x)\psi_R(x) $. $g$ is our candidate to realize the infimum defining $\rho(a)$. Note that
\begin{equation}
\label{contL2}
\nor{g}_2^2 = \int_{\R^d}  g_R^2(x) \psi_R^2(x) \, dx \le \int_{Q_R}  g_R^2(x)  \, dx \, .
\end{equation}

\begin{equation}
\label{contLp}
\nor{g}_{2p}^{2p}= \int_{\R^d}  \va{g_R(x)}^{2p} \psi_R^{2p}(x) \, dx 
\ge \int_{Q_R \setminus E_R}  \va{g_R(x)}^{2p}  \, dx  \ge 1 - 3^d R^{-d/2} \, ,
\end{equation}
where the last inequality comes from (\ref{cont.bord}).

Let us now estimate $\int_{\R^d}\abs{\mathcal{F}(g)(\omega)}^2 \abs{\omega}^\alpha d\omega$. To begin
 with, note that 
\begin{equation}
\label{def.grad}
\int_{\R^d}\abs{\mathcal{F}(g)(\omega)}^2 \abs{\omega}^\alpha d\omega
= c_{\alpha,d} \int_{\R^d} \int_{\R^d} \frac{(g(x)-g(y))^2}{\va{x-y}^{d+\alpha}} \, dx \, dy \, , 
\end{equation}
for some constant  $c_{\alpha,d} \in ]0;+\infty[$. Indeed, using Parseval's identity, 
\begin{eqnarray*}
\int_{\R^d} \int_{\R^d} \frac{(g(x)-g(y))^2}{\va{x-y}^{d+\alpha}} \, dx \, dy \,
& = &  \int_{\R^d} \frac{1}{\va{z}^{d+\alpha}}  \pare{ \int_{\R^d} (g(z+y)-g(y))^2 \, dy} \, dz \, .
\\
& = & \int_{\R^d} \frac{1}{\va{z}^{d+\alpha}}  \pare{ \int_{\R^d} \va{e^{- 2 \pi i \bra{z;\omega}} -1}^2 
\va{\FF(g)(\omega)}^2 \, d \omega} \, dz
\\
& = &  \int_{\R^d}  \va{\FF(g)(\omega)}^2  
\pare{ \int_{\R^d} \frac{4 \sin^2(\pi \bra{z;\omega}) }{\va{z}^{d+\alpha}} \, dz} \, d\omega
 \end{eqnarray*}  
Let $H(\omega) = \int_{\R^d} \frac{4 \sin^2(\pi \bra{z;\omega}) }{\va{z}^{d+\alpha}} \, dz$. For any orthogonal matrix 
 $O$ of $\R^d$, $H(O \omega)=H(\omega)$, so that 
\[ H(\omega) = H(\va{\omega} e_1)=  \va{\omega}^{\alpha}
\int_{\R^d} \frac{4 \sin^2(\pi z_1) }{\va{z}^{d+\alpha}} \, dz := \va{\omega}^{\alpha} c_{\alpha,d} \, ,
\]
where $c_{\alpha,d} \in ]0;\infty[$ for $\alpha \in ]0;2[$. This gives (\ref{def.grad}). 
Similarly, one obtains that 
\begin{equation}
\label{def.gradper}
R^d \sum_{z \in \Z^d} \va{\frac{z}{R}}^{\alpha} \abs{\mathcal{F}_R(g_R)(z)}^2 
= c_{\alpha,d} \int_{Q_R} \int_{\R^d} \frac{(g_R(x)-g_R(x+y))^2}{\va{y}^{d+\alpha}} \, dy \, dx \, , 
\end{equation}
for the same constant  $c_{\alpha,d}$ as in (\ref{def.grad}). Now, $\forall \delta > 0$,
\begin{eqnarray*}
&& \int_{\R^d} \int_{\R^d} \frac{(g(x)-g(y))^2}{\va{x-y}^{d+\alpha}} \, dx \, dy
 =  \int_{\R^d} \int_{\R^d} 
\frac{\cro{(g_R(x)-g_R(y)) \psi_R(x) + g_R(y) (\psi_R(x)-\psi_R(y))}^2}{\va{x-y}^{d+\alpha}} \, dx \, dy
\\
& & 
 \le  (1 + \delta) \int_{\R^d} \int_{\R^d} \frac{(g_R(x)-g_R(y))^2}{\va{x-y}^{d+\alpha}} \psi_R^2(x)\, dx \, dy
+ \frac{1 + \delta}{\delta} \int_{\R^d} \int_{\R^d} g_R^2(y) \frac{(\psi_R(x)-\psi_R(y))^2}{\va{x-y}^{d+\alpha}}
\, dx \, dy
\\
& & 
 \le  (1 + \delta)  \int_{Q_R} \int_{\R^d} \frac{(g_R(x)-g_R(x+z))^2}{\va{z}^{d+\alpha}} \, dz \, dx
+ \frac{1 + \delta}{\delta} \int_{\R^d} \int_{\R^d} g_R^2(y) \frac{(\psi_R(x)-\psi_R(y))^2}{\va{x-y}^{d+\alpha}}
\, dx \, dy \, .
\end{eqnarray*}
But, $\forall x, y \in \R^d$, $\va{\psi_R(x)-\psi_R(y)} \le \min(\sqrt{\frac{d}{R}} \va{x-y},1) (\ind_{Q_R}(x)
+ \ind_{Q_R}(y))$. This leads to 
 \begin{eqnarray*} 
 && \int_{\R^d} \int_{\R^d} g_R^2(y) \frac{(\psi_R(x)-\psi_R(y))^2}{\va{x-y}^{d+\alpha}}
\, dx \, dy \, 
\\
&& 
\le \int_{Q_R}  g_R^2(y) \pare{\int_{\R^d}  \frac{\min(\sqrt{\frac{d}{R}} \va{z},1)^2}{\va{z}^{d+\alpha}} dz} dy
+ 
\int_{Q_R} \int_{\R^d} g_R^2(x+z)  \frac{\min(\sqrt{\frac{d}{R}} \va{z},1)^2}{\va{z}^{d+\alpha}} \, dz \, dx
\\
&&
= 2 \pare{\int_{Q_R}  g_R^2(y) \, dy}  
 \pare{\int_{\R^d}  \frac{\min(\sqrt{\frac{d}{R}} \va{z},1)^2}{\va{z}^{d+\alpha}} dz} \, .
\end{eqnarray*}
It is easy to check that  there exists a constant $C$ ( only depending  on $\alpha$ and $d$) such that 
$\int_{\R^d} \frac{\min(\sqrt{\frac{d}{R}} \va{z},1)^2}{\va{z}^{d+\alpha}} dz \le C R^{-\alpha/2}$. It 
follows then from (\ref{eq:agR}),(\ref{def.grad}),  (\ref{def.gradper})  that  $\exists C$ such that 
$\forall \delta  > 0$, 
\begin{equation}
\label{cont.grad}
\int_{\R^d} \va{\omega}^{\alpha} \va{\FF(g)(\omega)}^2 d \omega
\le (1+\delta) R^d \sum_{z \in \Z^d} \va{\frac{z}{R}}^{\alpha} \va{\FF_R(g_R)(z)}^2 
+ C \frac{1+\delta}{\delta} \frac{I +\epsilon}{a} R^{-\alpha/2} \, .
\end{equation}
Combining (\ref{contL2}), (\ref{contLp}), (\ref{cont.grad}), we get that $\exists  C$ such that 
$\forall \epsilon > 0$, $\forall R_0 > 0$, ,$\exists R \ge R_0$ such that $\forall \delta > 0$, 
\begin{eqnarray*} 
& \inf & \acc{ a \nor{f}_2^2 + \int_{\R^d} \va{\omega}^{\alpha} \va{\FF(f)(\omega)}^2 d \omega ; \nor{f}_{2p}=1}
\\
& \le &  \frac{a \nor{g}_2^2 + \int_{\R^d} \va{\omega}^{\alpha} \va{\FF(g)(\omega)}^2 d \omega}{\nor{g}_{2p}^2}
\\
& \le & (1+\delta) \frac{a \nor{g_R}_{2,R}^2 + R^d \sum_{z \in \Z^d} 
		\va{\frac{z}{R}}^{\alpha} \va{\FF_R(g_R)(z)}^2}{1 - 3^d R^{-d/2}} 
		+ C \frac{1+\delta}{\delta} \frac{I + \epsilon}{a} \frac{R^{-\alpha/2}}{1 - 3^d R^{-d/2}} 
\\
& \le & (1+\delta) \frac{I + \epsilon}{1 - 3^d R^{-d/2}} + C \frac{1+\delta}{\delta} 
\frac{I + \epsilon}{a} \frac{R^{-\alpha/2}}{1 - 3^d R^{-d/2}} 
\end{eqnarray*}
This ends the proof of \refeq{lim.per}  by first letting $R_0$ go to $\infty$, and then $\delta$ and $\epsilon$ go
to $0$.  

To see that $\rho(a) = a^{1-d/(\alpha q)} \rho_{\alpha, d, p}$, use the transformation $g \mapsto g_{\lambda} =
\lambda^{d/(2p)} g(\lambda .)$ ($\lambda > 0$), and optimize over $\lambda$.
This ends the proof of Proposition \ref{OKCAVA4}.



 \section{Exponential moments lower bound.}
 \label{ExpLB}
The aim of this section is to prove the lower bound in \refeq{loglapNp}. This part is inspired by the proof of Theorem 1.3 in \cite{CL}. 
 Let us assume for a while the following theorem:
 \begin{theo}
 \label{TI}
 For any continuous function $f$ on $\R^d$ with compact support,
\begin{align*}
& \liminf_{t\rightarrow+\infty} \frac{\beta_t^\alpha}{t}\log \E\exp\pare{\beta_t^{-\alpha}\int_0^t f\pare{\frac{X_s}{\beta_t}}ds}\\
 \geq &\sup\acc{\int_{\R^d}f(x)g(x)^2dx-\int_{\R^d}\va{\omega}^\alpha\va{\F(g)(\omega)}^2d\omega,\nor{g}_2=1,g\geq 0}.
\end{align*}
 \end{theo}
It follows from $I_t=\beta_t^d\int_{\R^d}l_t(\lfloor \beta_t x\rfloor)^pdx$ that
\begin{align}
\label{1}
 \frac{\beta_t^\alpha}{t}\log \E\exp\pare{\theta \beta_t^{d/q-\alpha}N_p(l_t)}
 =\frac{\beta_t^\alpha}{t}\log \E\exp\pare{\theta\beta_t^{d-\alpha}\pare{\int_{\R^d} l_t(\lfloor \beta_tx\rfloor)^pdx}^{1/p}}.
 \end{align}
 Let $f$ be a continuous function with compact support such that $\nor{f}_{q}=1$. By H\"older inequality we have that:
 \begin{align*}
 \pare{\int_{\R^d} l_t(\lfloor \beta_tx\rfloor)^pdx}^{1/p} 
 &\geq \int_{\R^d}f(x)l_t(\lfloor\beta_t x\rfloor)dx\\
 &= \beta_t^{-d}\int_{\R^d} f\pare{\frac{x}{\beta_t}}l_t(\lfloor x\rfloor)dx\\
 &=\beta_t^{-d}\sum_{x\in\Z^d}  l_t(x) \int_{[x,x+1)^d} f\pare{\frac{y}{\beta_t}}dy .
 \end{align*}
 As $f$ is uniformly continuous and $\beta_t\rightarrow +\infty$,
 \begin{align}
\nonumber
 \pare{\int_{\R^d} l_t(\lfloor \beta_tx\rfloor)^pdx}^{1/p} 
 &\geq\beta_t^{-d}\sum_{x\in\Z^d} l_t(x)\pare{f\pare{\frac{x}{\beta_t}}+o(1)}\\
\label{2}
&= \beta_t^{-d}\pare{\int_0^tf\pare{\frac{X_s}{\beta_t}}ds+o(1)}.
 \end{align}
Combining (\ref{1}) and (\ref{2}),
 \begin{equation*}
  \frac{\beta_t^\alpha}{t}\log \E\exp\pare{\theta \beta_t^{d/q-\alpha}N_p(l_t)}
  \geq  \frac{\beta_t^\alpha}{t}\log \E\exp\pare{\beta_t^{-\alpha}\pare{\int_0^t \theta f\pare{\frac{X_s}{\beta_t}}ds+o(1)} }
 \end{equation*}
 Then taking the limit over $t$ and using Theorem \ref{TI}, 
 \begin{align*}
 & \liminf_{t\rightarrow +\infty}   \frac{\beta_t^\alpha}{t}\log  \E\exp\pare{\theta \beta_t^{d/q-\alpha}N_p(l_t)}
 \\
 & \hspace*{1cm}
 \geq \sup\acc{\theta \int_{\R^d}f(x)g(x)^2dx-\int_{\R^d}\va{\omega}^\alpha\va{\F(g)(\omega)}^2d\omega,\nor{g}_2=1,g\geq 0}
 .
 \end{align*}
 We take the supremum over all functions $f$ with compact support such that $\nor{f}_{q}=1$ and exchange the two supremum.  This
 immediately leads to
 \begin{equation}
 \label{BI}
 \liminf_{t\rightarrow +\infty}   \frac{\beta_t^\alpha}{t}\log \E\exp\pare{\theta \beta_t^{d/q-\alpha}N_p(l_t)}
 \geq \sup\acc{\theta \nor{g}_{2p}^2-\int_{\R^d}\va{\omega}^\alpha\va{\F(g)(\omega)}^2d\omega,\nor{g}_2=1,g\geq 0}. 
 \end{equation}

 Note that by \refeq{def.grad}, $\int_{\R^d}\va{\omega}^\alpha\va{\F(g)(\omega)}^2d\omega \ge 
 \int_{\R^d}\va{\omega}^\alpha\va{\F(\va{g})(\omega)}^2d\omega$, so that we can remove the 
 constraint $g \ge 0$ in the supremum  in \refeq{BI}. All that remains to show now  is 
 \begin{equation}
 \label{BIder}
   \sup\acc{\theta \nor{g}_{2p}^2-\int_{\R^d}\va{\omega}^\alpha\va{\F(g)(\omega)}^2d\omega,\nor{g}_2=1}
 = \pare{\frac{\theta}{\rho_{\alpha,d,p}}}^{\frac{\alpha q}{\alpha q-d}} \, .
 \end{equation}
 Using the tranformation $g \mapsto \lambda^{d/2} g(\lambda .)$ ($\lambda > 0$) and optimizing over $\lambda$ yields
 \begin{align*}
& \sup\acc{\theta \nor{g}_{2p}^2-\int_{\R^d}\va{\omega}^\alpha\va{\F(g)(\omega)}^2d\omega,\nor{g}_2=1}  
\\ 
& \hspace*{1cm}
= \frac{\alpha q -d}{\alpha q} \pare{\frac{d}{\alpha q}}^{\frac{d}{\alpha q-d}} \theta^{\frac{\alpha q}{\alpha q -d}}
\sup \acc{ \nor{g}_{2p}^{2 \frac{\alpha q}{\alpha q -d}} \pare{\int_{\R^d} \va{\omega}^{\alpha} \va{\F(g)(\omega)}^2 \, 
d \omega}^{- \frac{d}{\alpha q -d}}\, , \nor{g}_2 =1} . 
\end{align*}
The expression in the supremum above is invariant under the transformation $g \mapsto \lambda^{d/2} g(\lambda .)$. We can
therefore freely 
add the constraint $\nor{g}_{2p} =1$. This gives \refeq{BIder} and ends the proof of the lower bound in \refeq{loglapNp}.

\vspace{.5cm}

It remains to prove  Theorem \ref{TI}. \\
 {\em Proof of Theorem \ref{TI}.} \\
We first split the time interval $[0,t]$ into interval of length $\ent{\beta_t^\alpha}$. We set $\gamma_t=\ent{\frac{t}{\ent{\beta_t^\alpha}}}$.
\[ \E\exp\pare{\beta_t^{-\alpha}\int_0^t f\pare{\frac{X_s}{\beta_t}}ds}
 \geq  \exp\pare{-2\nor{f}_\infty} \E\exp\pare{\beta_t^{-\alpha} \int_{\ent{\beta_t^{\alpha}}}^{\ent{\beta_t^{\alpha}} \gamma_t} f\pare{\frac{X_s}{\beta_t}}ds}.
\] 
Let us introduce the two following operators. For any $\xi$ in $\ell^2(\Z^d)$ and for any $x\in\Z^d$:
\begin{align*}
&\Pi_t\xi (x)= \E_x\cro{\exp\pare{\beta_t^{-\alpha} \int_0^{\ent{\beta_t^\alpha}} f\pare{\frac{X_s}{\beta_t}}ds} \xi(X_{\ent{\beta_t^\alpha}})}
\\
&T_t\xi(x)= \E_x\cro{\exp\pare{\beta_t^{-\alpha}\int_0^{1}f\pare{\frac{X_s}{\beta_t} }ds}\xi(X_1) }
\end{align*}
$(X_t, t \ge 0)$ being symmetric,  $T_t$ is self-adjoint and the Markov property implies that $T_t^{\ent{\beta_t^\alpha}}=\Pi_t$.
 It follows that $\Pi_t$ is also self-adjoint. Now, let us introduce a non negative function $g\in\mathcal{C}^{\infty}_c(\R^d)$, the set of infinitely differentiable function with compact support, such that $\nor{g}_2=1$. We assume that the support of $g$ is included in 
 $[-M,M]^d$ and define $\xi_t$ by  $\xi_t(x)=\beta_t^{-d/2} g\pare{\frac{x}{\beta_t}}.$
\begin{align*}
&\E \cro{\exp\pare{\beta_t^{-\alpha} \int_{\ent{\beta_t^{\alpha}}}^{\ent{\beta_t^{\alpha}} \gamma_t} f\pare{\frac{X_s}{\beta_t}}ds}}\\
=& \sum_{x\in\Z^d} \P(X_{\ent{\beta_t^\alpha}} =x) \E_x \cro{\exp\pare{\beta_t^{-\alpha} \int_{0}^{\ent{\beta_t^{\alpha}} (\gamma_t-1)} f\pare{\frac{X_s}{\beta_t}}ds}}\\
\geq &\frac{1}{\sup\limits_{x\in\Z^d}\va{\xi_t(x)}^2} \sum_{x\in\Z^d} \P(X_{\ent{\beta_t^\alpha}} =x)\xi_t(x) 
\E_x\cro{\exp\pare{\beta_t^{-\alpha} \int_{0}^{\ent{\beta_t^{\alpha}} (\gamma_t-1)} f\pare{\frac{X_s}{\beta_t}}ds}\xi_t(X_{\ent{\beta_t^{\alpha}} (\gamma_t-1)})}\\
=& \frac{\beta_t^{d}}{\sup\limits_{x\in\Z^d}\va{g(x)}^2}\sum\limits_{x\in\Z^d} \P(X_{\ent{\beta_t^\alpha}} =x)\xi_t(x) \Pi_t^{\gamma_t-1}\xi_t(x)
\end{align*}
According to the  local limit theorem (Remark page 661 of Le Gall and Rosen \cite{LGR91}), 
$$\lim_{t\rightarrow +\infty} \sup\limits_{x\in\Z^d}\va{\beta_t^d \P(X_{\ent{\beta_t^\alpha}} =x)-p_1\pare{\frac{x}{\beta_t^d}}}=0,
$$
where $p_1$ is the transition density of the limit process. 
Since $\xi_t$ is supported by  $[-M\beta_t,M\beta_t]^d$, we only sum over this box. 
Moreover, there exists $\delta>0$ such that $p_1(x)>\delta$ for all $x\in[-M,M]^d$. 
$\xi_t$ being non negative, we get for any $\delta >0$ and any $t$ sufficiently large:
\begin{align*}
\E \cro{\exp\pare{\beta_t^{-\alpha} \int_{\ent{\beta_t^{\alpha}}}^{\ent{\beta_t^{\alpha}} \gamma_t} f\pare{\frac{X_s}{\beta_t}}ds}}
\geq \frac{\delta}{2 \sup\limits_{x\in\Z^d}\va{g(x)}^2} \sum\limits_{x\in\Z^d} \xi_t(x) \Pi_t^{\gamma_t-1}\xi_t(x)
=C\bra{\xi_t,\Pi_t^{\gamma_t-1}\xi_t}.
\end{align*}
Using the spectral representation of the operator $\Pi_t$, there exists a probability measure $\mu_{\xi_t}$ such that
$\bra{\xi_t,\Pi_t\xi_t}=\int_0^{+\infty}\lambda d\mu_{\xi_t}.$ It follows then from Jensen's inequality that
$$\bra{\xi_t,\Pi_t^{\gamma_t-1}\xi_t}
=\int_0^{+\infty}\lambda^{\gamma_t-1} d\mu_{\xi_t}
\geq \pare{\int_0^{+\infty}\lambda d\mu_{\xi_t}}^{\gamma_t-1}
=\bra{\xi_t,\Pi_t\xi_t}^{\gamma_t-1}.
$$
We have thus proved:
$$\liminf_{t\rightarrow+\infty} \frac{\beta_t^\alpha}{t}\log
 \E\cro{\exp\pare{\beta_t^{-\alpha}\int_0^t f\pare{\frac{X_s}{\beta_t}}ds}}
\geq \liminf_{t\rightarrow+\infty} \log \bra{\xi_t,\Pi_t\xi_t}.$$
Now, remember that $(X_t,t\geq 0)$ is in the domain of a stable process $(U_t,t\geq 0)$, i.e: $\frac{1}{t^{1/\alpha}}(X_{ts}, s \in 
\cro{0;1})\rightarrow (U_s, s \in [0;1])$ in the Skorokhod's J1 topology.
\begin{align*}
\bra{\xi_t,\Pi_t\xi_t}
& = \sum_{x\in\Z^d}\xi_t(x) \E_x\cro{\exp\pare{\beta_t^{-\alpha} \int_0^{\ent{\beta_t^\alpha}} f\pare{\frac{X_s}{\beta_t}}ds} \xi_t(X_{\ent{\beta_t^\alpha}})}\\
& =\beta_t^{-d} \sum_{x\in\Z^d} g\pare{\frac{x}{\beta_t}} \E_0 \cro{\exp\pare{ \int_0^1 f\pare{\frac{X_{s\ent{\beta_t^\alpha}}+x}{\beta_t}}ds} g\pare{\frac{X_{\ent{\beta_t^\alpha}}+x}{\beta_t}}}\\
& \underset{t\rightarrow +\infty}{\longrightarrow} \int_{\R^d}g(x) \E_0\cro{\exp\pare{\int_0^1f(U_s+x)ds}g(U_1+x)}dx\\
& =\int_{\R^d}g(x) \E_x\cro{\exp\pare{\int_0^1f(U_s)ds}g(U_1)}dx.
\end{align*}
Let us justify the convergence over $t$ above. Let $F$ be the functional, defined for any function $j$ by:
$$F(j)=\exp\pare{\int_0^1f\circ j(s)ds}g\circ j(1).$$
We have to prove that
$$\int_{\R^d}g(x) \E\cro{F\pare{(U_s+x,s\in[0,1])}}dx
-\beta_t^{-d} \sum_{x\in\Z^d} g\pare{\frac{x}{\beta_t}} \E\cro{F\pare{\frac{1}{\beta_t}\pare{X_{s\ent{\beta_t^\alpha}}+x,s\in[0,1]}}}
\underset{t\rightarrow +\infty}{\longrightarrow} 0.$$
Introducing 
$\int_{\R^d}g(x) \E\cro{F\pare{\frac{1}{\beta_t}X_{s\ent{\beta_t^\alpha}}+x,s\in[0,1]}}dx$
in the previous sum, we can apply the dominate convergence to obtain:
$$\int_{\R^d}g(x) \E\cro{F\pare{(U_s+x,s\in[0,1])}}dx
-\int_{\R^d}g(x) \E\cro{F\pare{\frac{1}{\beta_t}X_{s\ent{\beta_t^\alpha}}+x,s\in[0,1]}}dx
\underset{t\rightarrow +\infty}{\longrightarrow} 0.$$
Indeed, on one hand $F$ is continuous in the Skorokhod's J1 topology because $f$ and $g$ are continuous, and
on the other hand $F$ is bounded since $f$ and $g$ are bounded. Moreover we can see that
\begin{align*}
\int_{\R^d}g(x) \E\cro{F\pare{\frac{X_{s\ent{\beta_t^\alpha}}}{\beta_t}+x,s\in[0,1]}}dx
-&\beta_t^{-d} \sum_{y\in\Z^d} g\pare{\frac{y}{\beta_t}} 
\E\cro{F\pare{\frac{X_{s\ent{\beta_t^\alpha}}+y}{\beta_t},s\in[0,1]}}
\underset{t\rightarrow +\infty}{\longrightarrow} 0.
\end{align*}
Indeed, $f$ and $g$ being continuous with compact support, the function $x \in \R^d \mapsto 
F\pare{\frac{X_{s\ent{\beta_t^\alpha}}}{\beta_t}+x,s\in[0,1]}$ is uniformly continuous, and its modulus of 
continuity does not depend of $t$.  

At this point, we have proved that
\begin{align*}
\liminf_{t\rightarrow+\infty} \frac{\beta_t^\alpha}{t}\log \E \cro{\exp\pare{\beta_t^{-\alpha}\int_0^t f\pare{\frac{X_s}{\beta_t}}ds}}
&\geq  \log\int_{\R^d}g(x)\E_x\cro{\exp\pare{\int_0^1f(U_s)ds}g(U_1)}dx.\\
&=\log\bra{g,S_1g}
\end{align*}
where $S_t$ is the semigroup of operators defined on $L^2(\R^d)$ by:
$$S_th(x)=E_x\cro{\exp\pare{\int_0^tf(U_s)ds}h(U_t)}.$$
This semigroup is self-adjoint due to the symmetry of the process $(U_t,t\geq 0)$.
Thus, its infinitesimal operator $\mathcal{A}$ is also self-adjoint. Therefore, using the spectral representation of $\mathcal{A}$
and Jensen's inequality, we obtain:
\begin{align*}
\bra{g,S_1g}
& =\bra{g,\exp(\mathcal{A})g}
=\int_{\R}\exp(\lambda)d\mu_g(\lambda)
\\
& \geq \exp\int_{\R}\lambda d\mu_g(\lambda)
=\exp\bra{g,\mathcal{A}g} 
=\exp\pare{\int_{\R^d}f(x)g^2(x)dx-\int_{\R^d}\va{\omega}^\alpha\va{\F g(\omega)}^2d\omega}.
\end{align*}
Hence,
$$\liminf_{t\rightarrow+\infty} \frac{\beta_t^\alpha}{t}\log \E\cro{\exp\pare{\beta_t^{-\alpha}\int_0^t f\pare{\frac{X_s}{\beta_t}}ds}}
\geq \int_{\R^d}f(x)g(x)^2dx- \int_{\R^d}\va{\omega}^\alpha\va{\F(g)(\omega)}^2d\omega.$$
for any non negative function $g\in\mathcal{C}_c^{\infty}(\R^d)$. We now take the supremum over non negative functions 
$g \in\mathcal{C}_c^{\infty}(\R^d)$.  Using  representation (\ref{def.grad}) and boundedness of $f$, one can easily see
that this is the same as taking the supremum over non negative functions in $L^2(\R^d)$, and this ends the proof of 
 Theorem \ref{TI}.
\qed

\end{document}